\documentclass[twoside, reqno, 10pt]{amsart}

\usepackage[left=2.5cm, right=2.5cm, top=3cm, bottom=2.5cm]{geometry}

\usepackage[colorlinks=true, pdfstartview=FitV, linkcolor=blue,citecolor=blue, urlcolor=blue]{hyperref}

\usepackage[usenames]{xcolor} 

\definecolor{labelkey}{rgb}{0,0,1}
\definecolor{Red}{rgb}{0.7,0,0.1}
\definecolor{Green}{rgb}{0,0.7,0}

\usepackage{amsfonts, amssymb, amsmath, amsthm, mathrsfs, bbm, cjhebrew, gensymb, textcomp, mathtools, dsfont, calligra}

\usepackage[normalem]{ulem}

\usepackage{commath, setspace, subcaption, parcolumns, multirow, multicol, accents, comment, marginnote, verbatim, empheq, enumerate, stackrel, enumitem, float, physics}
\usepackage[capitalize,nameinlink,noabbrev]{cleveref}

\usepackage{graphicx, graphics, epsfig, psfrag, tikz, tikz-cd,svg}

\numberwithin{equation}{section}

\usepackage{tikz}

\makeatletter
\def\namedlabel#1#2{\begingroup
    #2%
    \def\@currentlabel{#2}%
    \phantomsection\label{#1}\endgroup
}
\makeatother

\newtheorem{Thm}{Theorem}[section]
\newtheorem{Lem}[Thm]{Lemma}
\newtheorem{Prop}[Thm]{Proposition}
\newtheorem{Cor}[Thm]{Corollary}

\newtheorem{Def}[Thm]{Definition}
\newtheorem{Rmk}[Thm]{Remark}

\newtheorem*{Thm*}{Theorem}

\newcommand{\N}{\mathbb{N}}
\newcommand{\Z}{\mathbb{Z}}

\newcommand{\R}{\mathbb{R}}

\newcommand{\T}{\mathbb{T}}

\newcommand{\B}{\mathcal{B}}
\newcommand{\NN}{\mathcal{N}}
\newcommand{\PP}{\mathcal{P}}

\newcommand{\QQ}{\mathcal{Q}}
\newcommand{\II}{\mathcal{I}}

\newcommand{\bk}{\mathbf{k}}

\DeclareMathOperator{\supp}{supp}
\DeclareMathOperator{\diam}{diam}

\newcommand{\h}[1]{\hat{#1}}

\newcommand{\til}[1]{{\tilde{#1}}}
\newcommand{\Sob}[2]{\lVert#1\rVert_{#2}}

\newcommand{\goesto}{\rightarrow}
\newcommand{\smod}{\setminus}

\newcommand{\al}{\alpha}
\newcommand{\be}{\beta}
\newcommand{\de}{\delta}
\newcommand{\De}{\Delta}
\newcommand{\gam}{\gamma}
\newcommand{\Gam}{\Gamma}
\newcommand{\eps}{\epsilon}
\newcommand{\veps}{\varepsilon}

\newcommand{\s}{\sigma}
\newcommand{\lam}{\lambda}
\newcommand{\kap}{\kappa}

\newcommand{\tht}{\theta}

\newcommand{\Tht}{\Theta}
\newcommand{\Om}{\Omega}

\newcommand{\bdy}{\partial}
\newcommand{\lb}{\langle}
\newcommand{\rb}{\rangle}

\begin{document}

\title[Mesh-free interpolant observables for data assimilation]{Mesh-free interpolant observables for continuous data assimilation}

\author{Animikh Biswas, Kenneth R. Brown, Vincent R. Martinez$^\dagger$}

\thanks{emails: abiswas@umbc.edu, kbr@ucdavis.edu, vrmartinez@hunter.cuny.edu$^\dagger$ \textit{corresponding author}}
\thanks{First published in Annals of Applied Mathematics in 2022, published by Global Science press}



\maketitle

\begin{abstract}
{This paper is dedicated to the expansion of the framework of general interpolant observables introduced by Azouani, Olson, and Titi for continuous data assimilation of nonlinear partial differential equations. The main feature of this expanded framework is its mesh-free aspect, which allows the observational data itself to dictate the subdivision of the domain via partition of unity in the spirit of the so-called Partition of Unity Method by Babuska and Melenk. As an application of this framework, we consider a nudging-based scheme for data assimilation applied to the context of the two-dimensional Navier-Stokes equations as a paradigmatic example and establish convergence to the reference solution in all higher-order Sobolev topologies in a periodic, mean-free setting. The convergence analysis also makes use of absorbing ball bounds in higher-order Sobolev norms, for which explicit bounds appear to be available in the literature only up to $H^2$; such bounds are additionally proved for all integer levels of Sobolev regularity above $H^2$.}
\end{abstract}

{\noindent \small {\it {\bf Keywords: continuous data assimilation, nudging, 2D Navier-Stokes equations, general interpolant observables, synchronization, higher-order convergence, partition of unity, mesh-free, Azounai-Olson-Titi algorithm}
  } 
  }

\section{Introduction}\label{sect:intro}

In recent years, several efforts have been made to develop a first-principles understanding of Data Assimilation (DA), where the underlying model dynamics are given by partial differential equations (PDEs) \cite{OlsonTiti2003, HaydenOlsonTiti2011, BlomkerLawStuartZygalakis2013, AzouaniOlsonTiti2014, CelikOlsonTiti2019, Sanz-AlonsoStuart2015, BessaihOlsonTiti2015, BranickiMajdaLaw2018, MondainiTiti18, IbdahMondainiTiti2019, BiswasFoiasMondainiTiti18}, as well to provide rigorous analytical and computational justification for its application and support for common practices therein, especially in the context of numerical weather prediction \cite{AltafTitiGebraelKnioZhaoMcCabeHoteit15, FarhatJollyTiti2015, MarkowichTitiTrabelsi2016, AlbanezLopesTiti2016, FoiasMondainiTiti16, FarhatLunasinTiti2016a, FarhatLunasinTiti2016b, FarhatLunasinTiti2016c, GeshoOlsonTiti2016, FarhatLunasinTiti2017, FarhatGlatt-HoltzMartinez2020, JollyMartinezTiti2017, LunasinTiti2017, JollyMartinezOlsonTiti2019, FarhatJohnstonJollyTiti18, LariosRebholzZerfas2019, HudsonJolly2019, LariosPei2020, BiswasBradshawJolly2020, BuzzicottiBonaccorsoDiLeoniBiferale2021}. A common representative model in these studies is the forced, two-dimensional (2D) Navier-Stokes equations (NSE) of an incompressible fluid, which contains the difficulty of high-dimensionality by virtue of being an infinite-dimensional, chaotic dynamical system, but whose long-time dynamics is nevertheless finite-dimensional, manifested, for instance, in the existence of a finite-dimensional global attractor  \cite{ConstantinFoias1988, FoiasManleyRosaTemam2001, TemamBook1997}. Given a domain $\Om\subset\R^2$, the 2D NSE is given by
    \begin{align}\label{eq:nse}
        \bdy_tu+(u\cdotp\nabla)u=-\nabla p+\nu\De u+f,\quad \nabla\cdotp u=0,
    \end{align}
supplemented with appropriate boundary conditions, where $u$ represents the velocity vector field, $\nu$ denotes the kinematic viscosity, $f$ is a time-independent, external driving force, $p$ represents the scalar pressure field. The underlying ideas in the works above, though originally motivated in large part by the classical problem of DA, that is, of reconstructing the underlying reference signal, has since been extended to the problem of parameter estimation; we refer the readers to the recent works \cite{CarlsonHudsonLarios2020, CarlsonHudsonLariosMartinezNgWhitehead2022, Martinez2022} for this novel application.

Central to the investigations of this paper is a certain algorithm for DA which synchronizes the approximating signal produced by the algorithm with the true signal corresponding to the observations. The algorithm of interest in this paper is a nudging-based scheme in which observational data of the signal is appropriately extended to the phase space of the system representing the truth, \eqref{eq:nse}. The interpolated data is then inserted into the system as an exogeneous term and is subsequently balanced through a feedback control term that serves to drive the approximating signal towards the observations. In particular, we consider the approximating signal to be given as a solution to the system  
    \begin{align}\label{eq:nse:ng}
        \bdy_tv+(v\cdotp\nabla)v=-\nabla q+\nu\De v+f-\mu (I_hv-I_hu),\quad \nabla\cdotp v=0,
    \end{align}
where $u$ represents a solution of \eqref{eq:nse} whose initialization is unknown, $I_hu$ represents observed values of the signal $u$, appropriately interpolated to belong the same phase space of solutions to \eqref{eq:nse}, $h$ quantifies the spatial density of the observations, and $\mu$ is a tuning parameter, often referred to as the ``nudging" parameter. The algorithm then consists of initializing \eqref{eq:nse:ng} arbitrarily and integrating it forward. The remarkable property of \eqref{eq:nse:ng} is that although the feedback control $-\mu(I_hv-I_hu)$ only enforces synchronization towards the observations, full synchronization of the signals $v$ and $u$ is nevertheless asymptotically ensured. Indeed, this property is conferred through a nonlinear mechanism, referred to as a Foias-Prodi property of determining values in the context of the 2D NSE, that is inherent to the system itself \cite{FoiasProdi1967, JonesTiti1992a, JonesTiti1992b, FoiasTemam1984, CockburnJonesTiti97}; this mechanism asymptotically enslaves the small scale features of the solution to its large scale features in the sense that knowledge of the asymptotic convergence of the large scale features of the difference of two solutions automatically imply asymptotic convergence of its small scale features as well. Morally speaking, any system which possess this property ``asymptotic enslavement" of small scales to large scales would guarantee the success of the nudging-based algorithm.

The ``nudging algorithm" was originally introduced by Hoke and Anthes in \cite{hoke1976initialization} in 1976, for finite-dimensional systems of ordinary differential equations. In a seminal paper of Azouani, Olson, and Titi \cite{AzouaniOlsonTiti2014}, this nudging scheme was appropriately extended to the case of partial differential equations via the introduction of the ``interpolant observable operator," denoted by $I_h$ above. There, it was shown that for $\mu,h>0$ chosen appropriately, that $v$ and $u$ asymptotically synchronize in the topology of $H^1(\Om)$, that is, in the topology of square-integrable functions with square-integrable spatial derivatives. On the other hand, it was observed in the computational work of Gesho, Olson, and Titi \cite{GeshoOlsonTiti2016} the convergence, in fact, appeared to be occurring in stronger topologies, for instance in the uniform topology of $L^\infty(\Om)$. This phenomenon was analytically confirmed in
\cite{BiswasMartinez2017} in the setting of periodic boundary conditions, where the observational data was given in the form of Fourier modes. In this setting, it was furthermore shown that synchronization occurs in a far stronger topology, that of the \textit{analytic} Gevrey topology, which is characterized by a norm in which Fourier modes are exponentially weighted in wave-number, provided that sufficiently many Fourier modes are observed. A distinguished property of this norm is that its finiteness identifies a length scale below which the function experiences an exponential cut-off in wave-number, and thus, can be reasonably ignored by numerical computation. In the context of turbulent flows, this length scale is known as the \textit{dissipation length scale} and is directly related to the radius of spatial analyticity of the corresponding flow \cite{FoiasTemam1989, Kukavica1998, FoiasManleyRosaTemam2001, BiswasJollyMartinezTiti2014}. Hence, the result in \cite{BiswasMartinez2017} rigorously established that the nudging-based algorithm synchronizes the corresponding signals all the way down to this length scale.

The case of other forms of observations, e.g., volume element, nodal values, etc., was not, however, treated in \cite{BiswasMartinez2017}. One of the central motivations of this paper is to therefore address these remaining cases. In order to do so, we develop a modest, general analytical framework in the spirit of \cite{AzouaniOlsonTiti2014} that ultimately allows one to demonstrate higher-order synchronization for the nudging-based algorithm, namely, beyond the $H^1$--topology, and in particular, \textit{any} $L^2$--based Sobolev topology. This framework accommodates a significantly richer class of interpolant observable operators based on the notion of a \textit{local} interpolant observable operator, which effectively allows one to use \textit{any} mode of observation within \textit{any} local region of the domain. These local interpolants are then made global by introducing a smooth partition of unity that allows one to patch the various observations across the domain and interpolate them appropriately into the phase space of the system. Although partitions of unity were already considered in several previous works for the nudging-based algorithm \cite{AzouaniOlsonTiti2014, BessaihOlsonTiti2015, JollyMartinezTiti2017, JollyMartinezOlsonTiti2019}, the partitions of unity used there were fixed and explicit, while in this work, we directly introduce the partition of unity as an additional parameter. Indeed, the most attractive feature allowed by the framework developed here is that it liberates the observations from the situation conceived in \cite{AzouaniOlsonTiti2014} of being constrained by a given distribution of measurement devices across the domain. Moreover, the possibility of having different spatial densities of measurements across the domain is also accommodated by this framework. This, of course, corresponds to the situation where more spatial measurements are simply available in one region of the domain compared to others. We note that this construction is akin to the ``Partition Finite Element Method" introduced by Babuska and Melenk \cite{BabuskaMelenk1997}, where finite element spaces were generalized to be ``mesh-free" in an analogous way via partition of unity, thus imbuing them with a greater flexibility.  We also refer the reader to the recent results \cite{BiswasBradshawJolly2020} and \cite{JollyPakzad2021}. In the former work, the efficacy of the nudging-based algorithm in the situation of having observations available \textit{only} in a fixed subdomain is assessed. The latter work studies higher-order interpolation using finite-element interpolants over bounded domains and the solution produced by the subsequent nudging-based algorithm is compared to solutions obtained by direct-numerical simulation from a semi-discrete scheme.

In \cref{sect:prelim}, we introduce the functional setting in which we will work throughout the paper. Note that we will work \textit{exclusively} in the periodic setting; the case of other boundary conditions will be treated in a future work. In \cref{sect:GIOs}, we introduce the notion of ``local interpolant observable operators" and construct a ``globalization" of them via partition of unity. Their approximation properties are subsequently developed and several nontrivial examples are provided (see \cref{sect:examples}). We point out that due to the amount of flexibility allowed by this construction, a significant portion of this work is dedicated to organizing its salient properties and ultimately identifying the combinations of interpolating operators that ultimately ensure well-posedness of the nudging-based algorithm and the important synchronization property described above. Rigorous statements of the main results of the paper are then provided in \cref{sect:statements} followed by several remarks. In order to clarify the detailed relation between the structure of the interpolant operators and the system, we introduce hyperdissipation into the system. Of course, all of our results contain the original, non-hyperdissipative case. In fact, a \textit{key feature} of the results is that synchronization in higher-order Sobolev spaces can be guaranteed under essentially the same assumptions on $\mu,h$ as were made in \cite{AzouaniOlsonTiti2014}, i.e., the assumptions exhibit the same scaling in $\mu,h$. In \cref{sect:statements}, we further identify alternative structural assumptions one can make on the interpolation operators that allow one to consider different families of operators that ultimately lead to the synchronization property (see \cref{def:partition:mult}). The proofs of the main statements are provided in \cref{sect:proofs}. We point out that in order to properly quantify the assumptions on $\mu, h$ required by the analysis to guarantee higher-order convergence, it is crucial to identify absorbing ball estimates with respect to the corresponding higher-order norms. This is captured in \cref{thm:abs:ball:Hk}, which properly generalizes the bounds obtained in \cite{DascaliucFoiasJolly2005} for the radius of the absorbing ball of \eqref{eq:nse} with respect to the $H^2$--topology. Finally, various technical details related to well-posedness (see \cref{sect:app:wp:ng}) or regarding the various aforementioned examples introduced in \cref{sect:examples} (see \cref{sect:app:taylor} and \cref{sect:app:vol}) are relegated to the appendices.

\section{{Mathematical Preliminaries}}\label{sect:prelim}

The functional setting throughout this paper will be the space of periodic, mean-free, divergence-free functions over $\T^2=[0,2\pi]^2$. More precisely, let ${B}_{per}(\T^2)$ denote the Borel measureable functions over $\T^2$, which are $2\pi$-periodic a.e. in each direction $x,y$. We define the space of $2\pi$-periodic, square-integrable functions over $\T^2$ by
    \begin{align}\label{def:L2}
        L^2(\T^2):=\{\phi\in{B}_{per}(\T^2): \Sob{\phi}{L^2}<\infty\}, \quad \Sob{\phi}{L^2}^2:=\int_{\T^2}|\phi(x)|^2dx.
    \end{align}
For each $k>0$, we define the inhomogeneous Sobolev space, ${H}^k(\T^2)$ by
    \begin{align}\label{def:Hk}
        {H}^k(\T^2):=\{\phi\in L^2(\T^2):\Sob{\phi}{H^k}<\infty\},\quad \Sob{\phi}{H^k}^2:=\sum_{|\al|\leq k}\Sob{\bdy^\al \phi}{L^2}^2.
    \end{align}
The homogeneous Sobolev space is defined as
    \begin{align}\label{def:dotHk}
        \dot{H}^k(\T^2):=\{\phi\in L^2(\T^2):\Sob{\phi}{\dot{H}^k}<\infty\},\quad \Sob{\phi}{\dot{H}^k}^2:=\sum_{|\al|=k}\Sob{\bdy^\al \phi}{L^2}^2.
    \end{align}
By the Poincar\'e inequality, the topologies induced by \eqref{def:Hk} and \eqref{def:dotHk} are equivalent. In particular, we have
    \begin{align}\label{eq:Hk:equiv}
        c^{-1}\Sob{\phi}{H^k}\leq \Sob{\phi}{\dot{H}^k}\leq\Sob{\phi}{H^k},
    \end{align}
for some universal constant $c>0$. Also observe that when $k=0$, we have $L^2(\T^2)=H^0(\T^2)=\dot{H}^0(\T^2)$. Lastly, let us recall the elementary fact that each element in the homogeneous Sobolev space can be identified with a mean-free function belonging to the inhomogeneous Sobolev space (see \cite{BenyiOh2013}). We will henceforth assume that each element of $\dot{H}^k(\T^2)$ is mean-free over $\T^2$.

We additionally incorporate the divergence-free condition by defining, for each $k\geq0$, the solenoidal Sobolev spaces. Note that due to \eqref{eq:Hk:equiv}, it will suffice to consider only the homogeneous counterpart. In particular, let us define
    \begin{align}\label{def:dotHk:sol}
        \dot{H}_\s^k(\T^2)^2&:=\{v\in \dot{H}^k(\T)^2:\nabla\cdotp v=0\}.
    \end{align}
Throughout the paper, we will often drop the notation, $\T^2$, for the domain when referring to the spaces $L^2(\T^2)$, $H^k(\T^2)$, or $\dot{H}^k(\T^2)$. Also, since the solenoidal distinction, $\s$, always refers to planar vector fields, we will simply write $\dot{H}_\s^k$ instead of $(\dot{H}_\s^k)^2$.
    
To analytically study \eqref{eq:nse}, it is customary to project \eqref{eq:nse} onto divergence-free vector fields and consider $\dot{H}_\s^k$ as the phase space of the resulting system. To do so, we introduce the Leray projection, $P_\s:(L^2)^2\goesto (L^2_\s)^2$, where $(L^2_\s)^2$ denotes the space of $L^2$--vector fields, $v$, such that $\nabla\cdotp v=0$ in the distributional sense, through its Fourier transform by
    \begin{align}\label{def:leray}
        \widehat{P_\s v}(k)=\sum_{j=1}^2\left[\frac{1}2\hat{ v}(k)-\frac{k}{|k|^2}k_j\hat{v}^j(k)\right],\quad k\in\Z^2\smod\{0\},
    \end{align}
and $\widehat{P_\s}v(0)=0$. We will consider a ``hyperdissipative" perturbation of \eqref{eq:nse}, which, for $p\geq0$ and $\gam>0$, is given by
    \begin{align}\label{eq:nse:proj:hyper}
        \bdy_tu-\nu\De u+\gam(-\De)^{p+1}u+P_\s(u\cdotp\nabla) u=P_\s f,\quad P_\s u=u,
    \end{align}
where $(-\De)^q$ denotes the operator defined by $\widehat{(-\De)^q\phi}(k)=|k|^{2q}\hat{\phi}(k)$, whenever $k\notin\Z^2$ and $q\geq0$. {We point out that while this modification of Navier-Stokes is not physical, it is common practice to consider $\gam,p>0$ in order to help stabilize numerical simulations. We consider this form of the dissipation in order to highlight the role of the dissipation in establishing the synchronization property of the nudging-based scheme at higher levels of Sobolev regularity.} The corresponding nudged system is then given by
    \begin{align}\label{eq:nse:ng:proj:hyper}
        \bdy_tv-\nu\De v+\gam(-\De)^{p+1}v+P_\s(v\cdotp\nabla)v=P_\s f-\mu P_\s I_h(v-u),\quad P_\s v=v.
    \end{align}
Given a solution $u$ of \eqref{eq:nse:proj:hyper} or solution $v$ of \eqref{eq:nse:ng:proj:hyper}, the pressure field may then be reconstructed up to an additive constant  \cite{ConstantinFoias1988, Temam2001}. For the remainder of the manuscript, we will consider the study of the coupled system \eqref{eq:nse:proj:hyper}, \eqref{eq:nse:ng:proj:hyper}. Note that, as with the Sobolev spaces, we will also abuse notation by writing $(L^2_\s)^2$ simply as $L^2_\s$.

The global well-posedness of \eqref{eq:nse:proj:hyper} in $\dot{H}^k$ and the existence of an absorbing ball in the corresponding topology are classical results and can be found, for instance, in \cite{ConstantinFoias1988, FoiasManleyRosaTemam2001, Temam2001}. When $k=2$, the sharpest estimate for the radius of the absorbing ball is established in  \cite[Theorem 3.1]{DascaliucFoiasJolly2005}. 
To state them, let us also recall the Grashof number, $G$, corresponding to a given time-independent external forcing, $f$, which is defined by
    \begin{align}\label{def:grashof}
        G:=\frac{\Sob{P_\s f}{L^2}}{\nu^2{\lam_1}},
    \end{align}
{where $\lam_1$ is the smallest eigenvalue of $-\De$. Since the side-length of the spatial domain has been normalized to have length $2\pi$, we see that $\lam_1=1$. In particular, $G$ is a non-dimensional quantity.}
Let us also define the following shape factors of the forcing. For $k\geq0$, We define the $k$--th order shape function of $f$ by
    \begin{align}\label{def:shape}
        \s_k:=\frac{\Sob{P_\s f}{\dot{H}^k}}{\Sob{P_\s f}{L^2}}.
    \end{align}
Observe that $\s_k\geq1$.

\begin{Prop}\label{prop:wp:nse}
Let $\gam,p\geq0$. Given $k\geq1$, let $f\in \dot{H}^{k-1}_\s$ and $u_0\in\dot{H}^k_\s$. There exists a unique $u$ satisfying  \eqref{eq:nse:proj:hyper} such that for all $T>0$, $u\in C([0,T];\dot{H}^k_{\s})\cap L^2(0,T;\dot{H}^{k+1})$ and $\frac{du}{dt}\in L^2(0,T;\dot{H}^{k-1}_\s)$. Moreover, there exists $t_0=t_0(u_0,f)$ such that 
    \begin{align}\label{est:energy:enstropy}
        \frac{\Sob{u(t)}{H^1}}{\nu}\leq 
        2G,
    \end{align}
for all $t\geq t_0$. Moreover, if $k\geq2$, then
   \begin{align}\label{est:palinstrophy}
       \frac{\Sob{\De u(t)}{L^2}}{\nu}\leq c_2(\s_1^{1/2}+G)G,
    \end{align}
for some universal constant $c_2>0$.
\end{Prop}

Lastly, let us recall the result proved in \cite{AzouaniOlsonTiti2014}, where synchronization in the $\dot{H}^1$--topology is shown for general interpolant observable operators, $I_h$, satisfying certain boundedness and approximation properties. For this, let us denote the $\dot{H}^1$--absorbing ball for \eqref{eq:nse:proj:hyper} by
    \begin{align}\label{def:abs:ball:H1}
        \B_1=\left\{v\in\dot{H}^1_\s:\Sob{v}{\dot{H}^1}\leq \sqrt{2}G\right\}.
    \end{align}
Moreover, assume that $I_h$ is finite-rank, linear, and satisfies either
    \begin{align}\label{def:Type1:og}
        \Sob{\phi-I_h\phi}{L^2}^2\leq c_1h^2\Sob{\phi}{\dot{H}^1}^2+c_2h^4\Sob{\phi}{\dot{H}^2}^2,
    \end{align}
or
    \begin{align}\label{def:Type2:og}
        \Sob{\phi-I_h\phi}{L^2}\leq c\Sob{\phi}{\dot{H}^1}.
    \end{align}
Although it was only proved for the unperturbed case, $\gam=0$, i.e., without hyperviscosity, we point out that the analysis of \cite{AzouaniOlsonTiti2014} still applies to the $\gam\neq0$ case without any difficulty whatsoever.

\begin{Thm}\label{thm:main:AOT}
Given $\gam,p\geq0$, $f\in L^2$, and $u_0\in\B_1$, let $u$ denote the unique solution corresponding $u_0, f$ guaranteed by \cref{prop:wp:nse}. Given $v_0\in\dot{H}^1_\s$, there exists a unique, $v$, satisfying \eqref{eq:nse:ng:proj:hyper} such that for all $T>0$, $v\in C([0,T];\dot{H}^1_\s)\cap L^2(0,T;\dot{H}^{2})\cap L^2(0,T;\dot{H}^{2+p})$ and $\bdy_tv\in L^2(0,T;L^2)$ provided that $\mu, h$ satisfy
    \begin{align}\label{cond:mu:wp:H1}
        c_0\frac{\mu h^2}{\nu}\leq 1,
    \end{align}
for some universal constant $c_0>0$. Moreover, one has
    \begin{align}\label{eq:sync:H1}
        \Sob{v(t)-u(t)}{\dot{H}^1}\leq e^{-(\mu/2)t}\Sob{v_0-u_0}{\dot{H}^1},
    \end{align}
provided that $\mu$ additionally satisfies
    \begin{align}\label{cond:mu:sync:H1}
        \mu\geq c_0'\nu(1+\log(1+G))G,
    \end{align}
for some universal constant $c_0'>0$.
\end{Thm}

In the next section, we expand upon the framework of general interpolation observable operators considered in \cite{AzouaniOlsonTiti2014} in order to accommodate approximation in higher-order Sobolev topologies. The specific examples of piecewise constant interpolation, volume element interpolation, and spectral interpolation constitute the original inspiration for the identification of properties \eqref{def:Type1:og} and \eqref{def:Type2:og}. The framework developed here introduces an additional degree of flexibility for interpolating the data that not only realizes these three examples as special cases, but generates a wealth of new examples that were not treated in \cite{AzouaniOlsonTiti2014}.

\begin{Rmk}\label{rmk:dimensions}
Note that we choose to work in the dimensionless domain, $\T^2$, {rather than $[0,L]^2$, for the sake of convenience. Because of this choice, derivatives and domains are ultimately dimensionless}. In particular, throughout the paper velocities and viscosities carry only the physical units of $(\text{time})^{-1}$. One may, of course, re-scale variables accordingly to introduce a length scale commensurate with the linear size of the spatial domain. In doing so, all physical quantities will then recover their appropriate dimensions. 
\end{Rmk}

\section{Local Interpolant Operators and Globalizability}\label{sect:GIOs}
In \cite{AzouaniOlsonTiti2014}, a general class of interpolant operators was introduced that could be used to define the nudging-based equation \eqref{eq:nse:ng:proj:hyper} and ultimately establish asymptotic convergence of its solution to the corresponding solution of \eqref{eq:nse:proj:hyper} in the topology of $L^2$ or $H^1$. One of the main contributions of the present article is to identify a very general class of interpolant operators that allows one to ensure convergence in a stronger topology. In particular, we introduce a class of interpolant operators that generalizes the class introduced in \cite{AzouaniOlsonTiti2014} in such a way that accommodates higher-order interpolants by introducing an additional layer of flexibility in their design. When a collection of them are defined locally, subordinate to some open covering of the domain, and they satisfy suitable approximation properties, the family can then be patched together to form a global interpolant; this is one of the main constructions in this paper and is very much akin to the so-called Partition of Unity Method introduced by Babuska and Melenk in \cite{BabuskaMelenk1997}. 

In what follows, we develop basic properties of this more general class of interpolating operators. Firstly, we introduce the notion of a local interpolant operator corresponding to a given subdomain of a given \textit{order} and \textit{level}. We then demonstrate how to ``globalize" the construction to the entire domain via partition of unity subordinate to a given covering by subdomains. The main difficulties that arise in doing so are due to the fact that at each subdomain, different interpolant operators can be specified, namely, ones that correspond to different orders and levels. We must therefore systematically develop terminology that distills their salient properties and ultimately allows one to differentiate among the various possibilities of the construction. Then in the local-to-global analysis, the structure of the constants associated to each local interpolation operator must be carefully tracked.

We begin by introducing the notion of a ``$Q$-local interpolation observable operator," {(I.O.O.)} where $Q$ represents a given subdomain of $\T^2$. Note that in the following definition, $H^k(Q)$ or $\dot{H}^k(Q)$ need not subsume any boundary conditions as it did in the case $Q=\T^2$ that we defined earlier; to distinguish between periodic boundary conditions, we will make use of the notation $H^k_{per}(Q)$. When $Q=\T^2$, we maintain the convention of dropping the dependence on the domain, e.g., $H^k=H^k(\T^2)$. Throughout this section, we will refer to any subset of $Q\subset\T^2$ that is bounded, open, and connected as a \textit{subdomain} of $\T^2$.

\begin{Def}\label{def:IO:loc}
Let $m\geq0$ and $k\geq m+1$ be integers. Let $Q\subset\T^2$ be a subdomain and denote $h=\diam(Q)$. We say that $I^Q$
is a \textbf{$\mathbf{\text{Q}}$--local I.O.O. of order $\mathbf{\text{m}}$ at level $\mathbf{\text{k}}$} if $I^Q$ is defined on $H^{k}(Q)$, linear, finite-rank,
and whose complement, $Id-I^Q$, for all $0\leq\ell\leq m$, satisfies 
    \begin{align}\label{eq:IO:loc:approx}
        \Sob{\phi-I^Q\phi}{\dot{H}^\ell(Q)}^2
        &\leq \sum_{j=1}^{k-\ell}\eps_{\ell,j}(I^Q)^2h^{2j}\Sob{\phi}{\dot{H}^{\ell+j}(Q)}^2,
    \end{align}
for some non-negative constants $\eps_{\ell,j}(I^Q)$. 
We will refer to the constants given by $\{\eps_{\ell,j}(I^Q)\}$ as the \textbf{constants associated to $I^Q$}. 
We say that $I^Q$ \textbf{interpolates optimally at level $k$ over $\mathbf{\text{Q}}$} if $I^Q$ is also a $Q$-local I.O.O. of order $k'-1$ at level $k'$, for all $1\leq k'\leq k$. In this case, for all $0\leq\ell\leq k'-1$, we have
    \begin{align}\label{eq:IO:loc:optimal}
        \Sob{\phi-I^Q\phi}{\dot{H}^\ell(Q)}^2
        &\leq \eps_{\ell,k'}(I^Q)^2h^{2(k'-\ell)}\Sob{\phi}{\dot{H}^{k'}(Q)}^2,
    \end{align}
for all $1\leq k'\leq k$. We say that $I^Q$ is a \textbf{$Q$--local I.O.O. of order $m$ at all levels} if \eqref{eq:IO:loc:approx} holds for all $k\geq m+1$; in this case, we also say \textbf{at level $k=\infty$}.
\end{Def}

Given a bounded, open, connected set, $Q$, with finite diameter, $h=\diam(Q)>0$, we recall \cite[Lemma 4.5.3]{BrennerScottBook} that since $Q$--local I.O.O.'s have finite rank, the following inverse inequality always holds for all such operators of order $m$ at level $k$:
    \begin{align}\label{eq:IO:loc:bdd}
        \Sob{I^Q\phi}{\dot{H}^\ell(Q)}\leq ch^{\ell'-\ell}\Sob{I^Q\phi}{\dot{H}^{\ell'}(Q)},  \quad\text{for all}\quad  0\leq\ell'\leq \ell\leq m,
    \end{align}
whenever $\phi\in H^{k}(Q)$, for some constant $c>0$, depending on $\ell, k$, but independent of $h$.

\begin{Rmk}\label{rmk:IO:loc:convention}
Observe that if $I^Q$ is an $m$--th order local I.O.O. at level $k$, then it is also a local I.O.O. of order $m$ at level $k'$, for all $k'>k$, as well as a local I.O.O. of order $m'$ at level $k$, for all $m'<m$.  Indeed, one can simply ``de-alias" the matrix induced by the associated constants by setting the additional associated constants to simply be zero. On the other hand, one can also identify a \textbf{canonical representative} for a $Q$--local I.O.O. by letting $m_0$ be the largest integer $m$ such that $\sup_j\eps_{m,j}(Q)>0$ and $k_0$ be the smallest integer $k$ such that $\sup_\ell\eps_{\ell,k'}=0$, for all $k'>k$. In this case, we may set $I^Q=I_{m_0,k_0}^Q$ without any ambiguity. It will be convenient to exploit the flexibility in the terminology later on (see \cref{lem:generic}).
\end{Rmk}

\begin{Rmk}\label{rmk:constants}
 We will always associate an I.O.O. to a subdomain $Q$. It will thus be more convenient to denote the associated constants of $I^Q$ simply by $\eps_{\ell,j}(Q)$, rather than $\eps_{\ell,j}(I^Q)$. This convention will be enforced after \cref{def:IO:global} below.
\end{Rmk}

A key object in this paper is the patching together of a family of local interpolant operators to form a global one. This is done via partition of unity. Given $k\geq2$ and a covering $\QQ=\{Q_q\}$ by subdomains $Q_q\subset\T^2$, let us fix any family of functions $\Psi=\{\psi_q\}_q\subset C^k$ satisfying
    \begin{description}
        \item[\namedlabel{item:P1}{(P1)}] for each $q$, $\psi_q|_{Q_q}=1$ and $\supp\psi_q\subset\til{Q}_q$, for all $q$, where $\til{Q}_q= {Q}_q+B(0,\de_q)$, for some ${\de_q\in(0,2\pi)}$;
        \item[\namedlabel{item:P2}{(P2)}] {there exists an integer $\pi_0>0$} such that for all $q$, {$\til{Q}_q\cap\til{Q}_{q'}\neq\varnothing$} for at most {$\pi_0$} many $q'$; 
        \item[\namedlabel{item:P3}{(P3)}] $\sum_q\psi_q(x)=1$, for all $x\in\Om$;
        \item[\namedlabel{item:P4}{(P4)}] {for all $0\leq\ell\leq k$, there exists $c_\ell>0$ such that $\sup_{|\al|=\ell}\Sob{\bdy^\al\psi_q}{L^\infty}\leq c_\ell h_q^{-\ell}$}, where $h_q=\diam(\til{Q}_q)$;
        \item[\namedlabel{item:P5}{(P5)}] there exists $\de>0$ such that whenever $\supp\psi_q\cap\supp\psi_{q'}\neq\varnothing$, one has $\de^{-1}h_q\leq h_{q'}\leq \de h_q$.
    \end{description}
We refer to \ref{item:P5} as the \textbf{\textit{$\de$--adic condition}}. Indeed, this condition implies that all ``neighbors," $Q_{q'}$, of $Q_q$ have diameters equivalent to $Q_q$ up to the fixed multiplicative factors $\de,\de^{-1}$. We will refer to $\Psi$ as a \textbf{\textit{$\de$--adic, $C^k$--partition of unity subordinate to $\QQ$}}. If $\Psi$ additionally satisfies $\Psi\subset C^\infty(\Om)$ and \ref{item:P4} holding for all $k$, then $\Psi$ is a \textbf{$\de$--adic, \textit{$C^\infty$--partition of unity}}. {For the majority of the manuscript, it will be assumed that $\Psi$ satisfies \ref{item:P1}--\ref{item:P5}, so we} will simply refer to $\Psi$ as a partition of unity. Lastly, it will also be useful to have additional control on the diameters in the covering. For this, we say that $\QQ$ is a \textbf{\textit{uniform cover at scale} $\mathbf{\textit{h}}$} if there exists $h>0$ such that $\de h\leq h_q\leq \de^{-1} h$, for all $q$. 

{Before proceeding to define a global interpolant operator, let us establish two useful facts which are consequences of the various partition of unity assumptions. In particular, for the moment, we do not necessarily assume that $\Psi$ satisfies every property \ref{item:P1}--\ref{item:P5}.}

\begin{Lem}\label{lem:pou}
Let $\{f_q\}_q\subset L^2(\Om)$. Suppose that $\{\psi_q\}_q\subset L^\infty(\Om)$ satisfies  \ref{item:P2} and \ref{item:P5}, and $\Sob{\psi_q}{L^\infty}\leq \lam(h_q)$, for all $q$, for some monotonic, homogeneous function $\lam:[0,\infty)\goesto[0,\infty)$ of degree $\rho$. Then
  \begin{align}\label{est:pou}
         \int_\Om\left(\sum_q\psi_q(x)f_q(x)\right)^2dx\leq N (\max\{\de,\de^{-1}\})^\rho  \sum_q\lam(h_q)^2\Sob{f_q}{L^2(\supp\psi_q)}^2.
     \end{align}
\end{Lem}

\begin{proof}
By the Cauchy-Schwarz inequality, it follows that
  \begin{align}
         \int_\Om\left(\sum_q\psi_q(x)f_q(x)\right)^2
         &\leq \left(\sum_{q,q'}\int_\Om|\psi_q(x)||\psi_{q'}(x)|f_q(x)^2dx\right)^{1/2}\left(\sum_{q,q'}\int_\Om|\psi_q(x)||\psi_{q'}(x)|f_{q'}(x)^2dx\right)^{1/2}\notag\\
         &\leq\sum_{q}\left(\sum_{\supp\phi_q\cap\supp\psi_{q'}\neq\varnothing}\Sob{\psi_{q'}}{L^\infty}\right)\Sob{|\psi_q|^{1/2}f_q}{L^2(\Om)}^2\notag\\
         &\leq N (\max\{\de,\de^{-1}\})^\rho  \sum_q\lam(h_q)^2\Sob{f_q}{L^2(\supp\psi_q)}^2,\notag
     \end{align}
where we applied \ref{item:P2}, \ref{item:P5}, and the boundedness hypothesis of the $\psi_q$ in obtaining the final two inequalities.
\end{proof}
   
\begin{Lem}\label{lem:multiplicity}
Let $\phi\in L^1(\Om)$ such that $\phi\geq0$. Suppose that $\Psi$ satisfies \ref{item:P1}--\ref{item:P3}. Then
	\begin{align}\label{est:multiplicity}
			\frac{1}{\pi_0}\sum_q\int_{\til{Q}_q}\phi(x)dx\leq \int_\Om\phi(x)dx\leq \sum_q\int_{\til{Q}_q}\phi(x)dx,
	\end{align}
where $\pi_0$ is the constant from \ref{item:P2}.
\end{Lem}
\begin{proof}
{Let $\mathcal{N}_q$ denote the set of indices, $q'$, such that $\til{Q}_{q'}\cap Q_q\neq\varnothing$. Observe that since $Q_{q'}\subset\til{Q}_{q'}$, we have from \ref{item:P2} that $\#(\mathcal{N}_q)\leq N$, for all $q$. Since $\mathcal{Q}$ is a cover of $\Om$, it follows that $\til{Q}_q\subset\bigcup_{q'\in\mathcal{N}_q}Q_q$. Since $\psi_{q'}=1$ on $Q_{q'}$ by \ref{item:P1}, we deduce}
	\begin{align}
	\sum_q\int_{\til{Q}_q}\phi(x)dx\leq\sum_q\sum_{q'\in\mathcal{N}_q}\int_{{Q}_{q'}\cap\til{Q}_q}\phi(x)dx\leq\sum_{q'}\left(\sum_{q}\chi_{\mathcal{N}_q}(q')\right)\int_{Q_{q'}}\psi_{q'}(x)\phi(x)dx\notag.
	\end{align}
{Since $\sum_{q}\chi_{\mathcal{N}_q}(q')\leq\#(\mathcal{N}_{q'})\leq \pi_0$, for each $q'$, we may conclude from \ref{item:P3} that}
	\begin{align}\notag
	\sum_q\int_{\til{Q}_q}\phi(x)dx\leq \pi_0\int_\Om\phi(x)dx
	\end{align}
{On the other hand, since $\mathcal{Q}$ is a cover of $\Om$ and $Q_q\subset\til{Q}_q$ by \ref{item:P1}, it follows that}
	\begin{align}
		\int_\Om\phi(x)dx\leq \sum_q\int_{Q_q}\phi(x)dx\leq\sum_q\int_{\til{Q}_q}\phi(x)dx,\notag
	\end{align}
{which completes the proof.}
\end{proof}

\begin{Rmk}\label{rmk:finite}
Partitions of unity satisfying \ref{item:P1}--\ref{item:P5} were constructed in \cite{AzouaniOlsonTiti2014, BessaihOlsonTiti2015, JollyMartinezOlsonTiti2019}. There, a collection of augmented squares overlapped in a regular manner to cover the domain multiple times; one may refer to this property as having ``finite partition multiplicity." In general, the collection of open sets to which a partition of unity is subordinate, need not satisfy this property. Indeed, let us formally introduce this notion as follows:

\begin{Def}\label{def:partition:mult}
Let $\QQ=\{Q_q\}_q$ be a covering of $\Om$ by bounded, open, connected subsets. We say that $\QQ$ has \textbf{partition multiplicity}, ${M}>0$, if there exists an integer, ${M}>0$, and  subcollections $\QQ_1,\dots, \QQ_{M}\subset\QQ$ such that $\bigcup_{j=1}^M\QQ_j=\QQ$, $\bigcup_{Q\in\QQ_j}\bar{Q}=\Om$, and $|\bar{Q}\cap \bar{Q}'|=0$, for all $Q,Q'\in\QQ_j$, for all $j=1,\dots, M$, where $\bar{Q}$ denotes the closure of $Q$.
\end{Def}

\begin{Lem}\label{lem:partition:mult}
Suppose $\QQ$ is a covering of $\Om$ with partition multiplicity ${M}$. Then
    \begin{align}\label{est:partition:mult}
        \frac{1}{M}\sum_Q\int_Q\phi(x)dx\leq\int_\Om\phi(x)dx\leq\sum_{Q\in\QQ_j}\int_Q\phi(x)dx,
    \end{align}
for all $\phi\in L^1(\Om)$ such that $\phi\geq0$, and for all $j=1,\dots M$.
\end{Lem}

\begin{proof}
Let $\QQ_1,\dots, \QQ_{M}$ denote the ${M}$ subcollections of $\QQ$ that each cover $\Om$ and each of whose respective elements can overlap only on sets of zero Lebesgue measure. Observe that for all $j=1,\dots, M$, we have
    \begin{align}
        \int_\Om\phi(x)dx=\sum_{Q\in\QQ_j}\int_Q\phi(x)dx,\notag
    \end{align}
which implies the upper bound. Now, upon averaging in $j$ and applying Fubini's theorem, we arrive at
    \begin{align}\notag
        \int_\Om\phi(x)dx=\frac{1}{M}\sum_{j=1}^{M}\sum_{Q\in\QQ_j}\int_Q\phi(x)dx=\frac{1}{M}\sum_Q\int_Q\phi(x)dx\sum_{j=1}^{M}\chi_{\QQ_j}(Q).
    \end{align}
Since $\phi\geq0$ and $\sum_{j=1}^M\chi_{\QQ_j}(Q)\geq1$, it follows that 
    \begin{align}\notag
    \int_\Om\phi(x)dx
    \geq \frac{1}{M}\sum_Q\int_Q\phi(x)dx,
    \end{align}
which produces the lower bound, as desired.
\end{proof}

{We therefore see that the first inequality of \cref{lem:multiplicity} already follows from the assumptions \ref{item:P1}--\ref{item:P5} (see  \cref{lem:multiplicity}). Indeed, property \ref{item:P2} basically asserts a type of ``local multiplicity," whereas a cover with finite partition multiplicity is a form of ``global multiplicity." In contrast, the assumptions on $\Psi$ allow for the possibility of having an infinite open covering in the case of a general bounded domain, i.e., bounded, open, connected subset of the plane. Indeed, if $\Om$ is a disk centered at the origin, then the open covering given by a small disk centered at the origin followed by consecutively overlapping concentric open annuli with geometrically decreasing length, appropriately proportional to the radius of the disk, provides such an example.}
\end{Rmk}



Let us now define a global interpolant operator. For convenience, whenever we refer to a partition of unity, we will specifically consider ones of the type described above, that is, satisfying \ref{item:P1}--\ref{item:P5}.

\begin{Def}\label{def:IO:global}
Given a covering, $\QQ=\{Q_q\}_q$, of $\Om$ by bounded, open, connected subsets with $h_q=\diam(Q_q)$, we say that
the family, $\II=\{I^{(q)}\}_q$, of local I.O.O.'s is \textbf{subordinate to $\QQ$} if for each $q$, $I^{(q)}$ is an $m_q$--th order $Q_q$--local I.O.O. at level $k_q$, for some integers $m_q\geq0$ and $k_q\geq m_q+1$. We furthermore say that the family is \textbf{$\QQ$--uniform} if the associated constants of each $I^{(q)}\in\II$ satisfy $\sup_q\sup_{\ell\leq j\leq k_q-1}\eps_{\ell,j}(Q_q)<\infty$, for each $\ell$.  We say that $\II$ is an \textbf{$\mathbf{\text{(m,k)}}$--generic} family if there exist $m\geq0$ and $\infty\geq k\geq m+1$ such that $m=m_q$ and $k=k_q$, for all $q$.

Given $m\geq0$ and $m+1\leq k\leq \infty$, we say an operator $I_{m,k}$ is a 
\textbf{\textit{$(\II,\Psi)_\QQ$--subordinate global I.O.O. of order $\mathbf{\text{m}}$ at level $\mathbf{\text{k}}$}} if there exists a $\QQ$--subordinate family, $\II$, of I.O.O.'s, and $\QQ$--subordinate $C^k$--partition of unity, $\Psi$, such that $m\leq \inf_qm_q$ and $k\geq\sup_qk_q$, and
    \begin{align}\label{eq:IO:global}
        (I_{m,k}\phi)(x):=\sum_q\psi_q(x)(I^{(q)}\phi_q)(x),\quad x\in \Om,
    \end{align}
whenever $\phi\in \bigcap_{q}H^{k_q}(\Om)$, where $\phi_q=\phi|_{\til{Q}_q}$. 

\end{Def}

 Note that when the associated partition of unity is clear, we will simply say that $I_{m,k}$ is the $\II$--subordinate global I.O.O. with associated covering $\QQ$. On the other hand, in light of the ``dealiasing" procedure described in \cref{rmk:IO:loc:convention}, we see that if $I_{m,k}$ is an $(\II,\Psi)_\QQ$-subordinate global I.O.O., then each $I^{(q)}\in\II$ is an $m$--th order $Q_q$--local I.O.O. at level $k$ such that $m\leq\inf_qm_q$ and $k\geq\sup_qk_q$. In particular, we immediately deduce the following fact.

\begin{Lem}\label{lem:generic}
Let $m,k\geq0$ be such that $k\geq m+1$. Every $(\II,\Psi)_\QQ$--subordinate global I.O.O. of order $m$ at level $k$ may be realized as an $(\til{\II},\Psi)_\QQ$--subordinate global I.O.O. of order $\til{m}$ at level $\til{k}$, where $\til{\II}$ is $(\til{m},\til{k})$--generic. In particular, we may choose $\til{m}=\inf_qm_q$ and $\til{k}=\sup_qk_q$, where $(m_q,k_q)$ denotes the order and level associated to the canonical representative of $I^{(q)}$.
\end{Lem}

Without loss of generality, we may therefore always assume that any global I.O.O., $I_{m,k}$, derives from an $(m,k)$--generic family $\II$ of local I.O.O.'s. Now, as a consequence of \cref{def:IO:loc}, the properties of the partition of unity, and \eqref{eq:IO:global}, we have the following.

\begin{Prop}\label{prop:IO:global}
Let $m,k\geq0$ be integers such that $k\geq m+1$. Let $I_{m,k}$ be an $(\II,\Psi)_\QQ$--subordinate global I.O.O, where $\II$ is $(m,k)$--generic, and $\QQ=\{Q_q\}_q$ denotes the associated covering. Then there exist constants $\{\veps_{\ell,j}(Q_q)\}_q$ such that for all $0\leq\ell\leq m$
    \begin{align}\label{est:IO:Hl}
     \Sob{\phi-I_{m,k}\phi}{\dot{H}^\ell}^2
         &\leq c\sum_{j=1}^{k}\sum_q\veps_{\ell,j}(Q_q)^2h_q^{2(j-\ell)}\Sob{\phi}{\dot{H}^{j}(\til{Q}_q)}^2,
    \end{align}
for some constant $c>0$, independent of $h$, and where $\veps_{\ell,j}$ can be specified as
      \begin{align}\label{def:IO:const:gen}
        \veps_{\ell,j}(Q_q)^2:=
        \sum_{i=0}^{j-1}\eps_{i,j-i}(Q_q)^2,
    \end{align}
where $\eps_{i,j}(Q_q)$ are the constants associated to $I^{(q)}\in\II$. On the other hand, if $I^{(q)}$ interpolates optimally over $Q_q$ (at level $k$), for all $q$, then for all $1\leq k'\leq k$ and $0\leq\ell\leq k'-1$
    \begin{align}\label{est:IO:optimal}
        \Sob{\phi-I_{k}\phi}{\dot{H}^\ell}^2\leq \sum_q\veps_{\ell,k'}(Q_q)^2h_q^{2(k'-\ell)}\Sob{\phi}{\dot{H}^{k'}(\til{Q}_q)}^2,
    \end{align}
where
    \begin{align}\label{def:IO:const:optimal}
        \veps_{\ell,k'}(Q_q)^2=\sum_{i\leq\ell}\eps_{i,k'-i}(Q_q)^2.
    \end{align}
\end{Prop}

\begin{proof}
Since $\Psi$ is a partition of unity, observe that
    \begin{align}
        \phi(x)- I_{m,k}\phi(x)=\sum_q\psi_q(x)(\phi(x)-I^{(q)}\phi(x)).\notag
    \end{align}
Let $\al$ be a multi-index such that $|\al|=\ell$, where $0\leq\ell\leq m$. It then follows from the Leibniz rule that
    \begin{align}
        \bdy^\al(\phi-I_{m,k}\phi)=\sum_q\sum_{\be\leq\al}c_{\al,\be}\bdy^{\al-\be}\psi_q(x)\bdy^{\be}(\phi-I^{(q)}\phi).\notag
    \end{align}
Upon taking absolute values, squaring both sides, integrating over $\Om$, summing over $|\al|\leq\ell$, then applying \eqref{eq:IO:loc:approx}, \ref{item:P4}, and \cref{lem:pou} (with $\phi_q=\bdy^{\al-\be}\psi_q$), we obtain
    \begin{align}
        \Sob{\phi-I_{m,k}\phi}{\dot{H}^\ell}^2&\leq c\sum_{|\be|\leq|\al|}c_{\al,\be}\sum_qh_q^{2(i-\ell)}\Sob{\phi-I^{(q)}\phi}{\dot{H}^{|\be|}(\til{Q}_q)}^2\label{est:global:gen:approach}\\
        &\leq c\sum_q\sum_{i\leq\ell}\sum_{j=1}^{k-i}h_q^{2(i+j-\ell)}\eps_{i,j}(Q_q)^2\Sob{\phi}{\dot{H}^{i+j}(\til{Q}_q)}^2\notag\\
         &= c\sum_q\sum_{i\leq\ell}\sum_{j=i+1}^{k}h_q^{2(j-\ell)}\eps_{i,j-i}(Q_q)^2\Sob{\phi}{\dot{H}^{j}(\til{Q}_q)}^2,\notag
    \end{align}
where we shifted indices to obtain the last inequality. Finally, changing the order of summation yields
    \begin{align}\label{est:global:gen:final}
         \Sob{\phi-I_{m,k}\phi}{\dot{H}^\ell}^2
         &\leq c\sum_{j=1}^{k}\sum_q\left(h_q^{2(j-\ell)}\veps_{\ell,j}(Q_q)^2\Sob{\phi}{\dot{H}^{j}(\til{Q}_q)}^2\right),
    \end{align}
which is \eqref{est:IO:Hl}. 

On the other hand, if $I_{m,k}=I_k$ interpolates optimally, then for $1\leq k'\leq k$ and $0\leq\ell\leq k'-1$, we apply \eqref{eq:IO:loc:optimal} in \eqref{est:global:gen:approach}, then \ref{item:P4} and \cref{lem:pou}, as before, to obtain
    \begin{align}\notag
        \Sob{\phi-I_{k}\phi}{\dot{H}^\ell}^2\leq c\sum_q\left(\sum_{i\leq\ell}\eps_{i,k'-i}(Q_q)^2\right)h_q^{2(k'-\ell)}\Sob{\phi}{\dot{H}^{k'}(\til{Q}_q)}^2,
    \end{align}
which is \eqref{est:IO:optimal}, as desired.
\end{proof}

In light of \cref{prop:IO:global}, we may define the following terminology.

\begin{Def}\label{def:IO:global:uniform}
 Let $I_{m,k}$ be an $(\II,\Psi)_\QQ$--subordinate global I.O.O. such that $\II$ is $(m,k)$--generic. We say that $I_{m,k}$ is \textbf{$\QQ$--uniform} if $\II$ is a $\QQ$--uniform family. If $I_{m,k}$ is $\QQ$--uniform and $\QQ$ is a uniform cover at scale $h$, then we say that $I_{m,k}$ \textbf{interpolates uniformly at scale $\mathbf{\text{h}}$}. If $I^{(q)}\in\II$ interpolates optimally over $Q_q$ at level $k$, for all $q$, then we say $I_{m,k}$ \textbf{interpolates optimally} and denote it simply as $I_k$.
\end{Def}

From \cref{def:IO:loc} and \eqref{eq:IO:global}, one also easily obtains as a corollary to \cref{prop:IO:global} and \cref{lem:multiplicity}, the following interpolation error estimates for various special cases.

\begin{Cor}\label{cor:IO:global}
Let $m,k\geq0$ be integers such that $k\geq m+1$. Let $I_{m,k}$ be an $(\II,\Psi)_\QQ$--subordinate global I.O.O, where $\II$ is $(m,k)$--generic, and $\QQ=\{Q_q\}_q$ denotes the associated covering. If $I_{m,k}$ is $\QQ$--uniform, then for all $0\leq \ell\leq m$
    \begin{align}\label{est:IO:Quniform}
        \Sob{\phi-I_{m,k}\phi}{\dot{H}^\ell}^2\leq c\sum_{j=1}^{k}\veps_{\ell,j}^2\sum_qh_q^{2(j-\ell)}\Sob{\phi}{\dot{H}^j(\til{Q}_q)}^2.    
    \end{align}
If $\QQ$ is a uniform cover at scale $h$, then for all $0\leq\ell\leq m$
    \begin{align}\label{est:IO:atscale}
        \Sob{\phi-I_{m,k}\phi}{\dot{H}^\ell}^2\leq c\sum_{j=1}^{k}h^{2(j-\ell)}\sum_q\veps_{\ell,k}(Q_q)^2\Sob{\phi}{\dot{H}^{j}(\til{Q}_q)}^2,
    \end{align}
In particular, if $I_{m,k}$ interpolates uniformly at scale $h$, then for all $0\leq\ell\leq m$ 
    \begin{align}\label{est:IO:regularly}
   \Sob{\phi-I_{m,k}\phi}{\dot{H}^\ell}^2\leq c\sum_{j=1}^{k}h^{2(j-\ell)}\Sob{\phi}{\dot{H}^{j}}^2;
    \end{align}
if, additionally, $I_{m,k}=I_{k}$ interpolates optimally, then for all $0\leq\ell\leq k'-1$ and $1\leq k'\leq k$
    \begin{align}\label{est:IO:regularly:optimally}
        \Sob{\phi-I_{k}}{\dot{H}^\ell}^2\leq ch^{2(k'-\ell)}\Sob{\phi}{\dot{H}^{k'}}^2.
    \end{align}
\end{Cor}

Lastly, from \cref{prop:IO:global}, we also immediately deduce the following boundedness property of global I.O.O.'s.

\begin{Cor}\label{cor:IO:global:bdd}
Let $m,k\geq0$ such that $k\geq m+1$ and $I_{m,k}$ be an $(\II,\Psi)_\QQ$--subordinate global I.O.O. If $I_{m,k}$ is $\QQ$--uniform, then there exists a constant $c>0$ such that
    \begin{align}\label{eq:IO:global:bdd}
            \Sob{I_{m,k}\phi}{\dot{H}^\ell}\leq c\Sob{\phi}{\dot{H}^{k}},\quad \ell=0,1.
    \end{align}
If, moreover, $I_{m,k}$ interpolates uniformly at scale $h$, then
    \begin{align}\label{eq:IO:global:bdd:uniform}
        \Sob{I_{m,k}\phi}{\dot{H}^\ell}\leq ch^{-\ell}\begin{cases}h\Sob{\phi}{\dot{H}^{k}},&2\leq\ell\leq m\\
        \Sob{\phi}{\dot{H}^k},&m<\ell\leq k.
        \end{cases}
    \end{align}
where $c$ is independent of $\QQ$. In particular, if $I_{m,k}$ interpolates uniformly at scale $h$, then $I_{m,k}:\dot{H}^{k}\goesto\dot{H}^{k}$, is a bounded operator, for all $k\geq m+1$, where $m\geq0$.
\end{Cor}

\begin{proof}
Suppose $0\leq \ell\leq m$. By the triangle inequality and \eqref{est:IO:Quniform}, we have
    \begin{align}
        \Sob{I_{m,k}\phi}{\dot{H}^\ell}^2&\leq 2\Sob{\phi}{\dot{H}^\ell}^2+2\Sob{\phi-I_{m,k}\phi}{\dot{H}^\ell}^2\leq c\Sob{\phi}{\dot{H}^\ell}^2+2\sum_{j=1}^{k}\veps_{\ell,j}^2\sum_qh_q^{2(j-\ell)}\Sob{\phi}{\dot{H}^j(\til{Q}_q)}^2.\label{est:bdd:Quniform}
    \end{align}
In the particular case $m\geq\ell$ and $\ell=0,1$, we may apply the fact that $0< h_q\leq2\pi$, for all $q$, \cref{lem:multiplicity}, and Poincar\'e's inequality to deduce
    \begin{align}\notag
    \Sob{I_{m,k}\phi}{\dot{H}^\ell}^2\leq c\sum_{j=1}^{k}\sum_q\Sob{\phi}{\dot{H}^j(\til{Q}_q)}^2\leq c\Sob{\phi}{\dot{H}^k}^2.
    \end{align}
This establishes \eqref{eq:IO:global:bdd}.

Now, if $I_{m,k}$ interpolates uniformly at scale $h$, then \eqref{est:bdd:Quniform}, \cref{lem:multiplicity}, and Poincar\'e's inequality imply
    \begin{align}
        \Sob{I_{m,k}\phi}{\dot{H}^\ell}^2\leq 2\Sob{\phi}{\dot{H}^\ell}^2+c\sum_{j=1}^{k}h^{2(j-\ell)}\sum_q\Sob{\phi}{\dot{H}^j(\til{Q}_q)}^2\leq ch^{2(1-\ell)}\Sob{\phi}{\dot{H}^k}^2.\notag
    \end{align}
This establishes boundedness from $\dot{H}^k\goesto\dot{H}^\ell$, for all $0\leq\ell\leq m$.

Now suppose that $m<\ell\leq k$ and $I_{m,k}$ interpolates uniformly at scale $h$. By the product rule, \cref{lem:pou} (with $\phi_q=\bdy^\gam\psi_q$, $\gam=\al-\be$, $|\al|=\ell$, $|\be|=i$), and \eqref{eq:IO:loc:approx}, we deduce
    \begin{align}
        \Sob{I_{m,k}\phi}{\dot{H}^\ell}^2
        &\leq c\sum_{i\leq \ell}\sum_qh_q^{2(i-\ell)}\Sob{I^{(q)}\phi}{\dot{H}^i(\til{Q}_q)}^2\notag\\
        &\leq ch^{-2\ell}\sum_{i\leq m}h^{2i}\sum_q\Sob{I^{(q)}\phi}{\dot{H}^i(\til{Q}_q)}^2+ch^{-2\ell} h^{2m}\sum_q\Sob{I^{(q)}\phi}{\dot{H}^m(\til{Q}_q)}^2\label{est:Imk:Hl:bdd}.
    \end{align}
For the first sum in \eqref{est:Imk:Hl:bdd}, we apply the facts that $I^{(q)}:\dot{H}^k(Q_q)\goesto\dot{H}^i(Q_q)$ boundedly, which was just proved, and that $\QQ$ is a uniform cover at scale $h$. For the second, we use the triangle inequality, \eqref{eq:IO:loc:approx}, and the fact that $I^{(q)}$ interpolates at scale $h$, for all $q$, to obtain
    \begin{align}
       \Sob{I^{(q)}\phi}{\dot{H}^m(\til{Q}_q)}^2&\leq 2\Sob{\phi}{\dot{H}^m(\til{Q}_q)}^2+2\Sob{\phi-I_{m,k}^{(q)}\phi}{\dot{H^m}(\til{Q}_q)}^2\notag\\
       &\leq 2\Sob{\phi}{\dot{H}^m(\til{Q}_q)}^2+2\sum_{j=1}^{k-m}\eps_{m,j}(Q_q)^2h^{2j}\Sob{\phi}{\dot{H}^{m+j}(\til{Q}_q)}^2.\notag
    \end{align}
We then apply $\QQ$--uniformity, sum over $q$, and apply \cref{lem:multiplicity} to deduce
    \begin{align}
        \sum_q \Sob{I^{(q)}\phi}{\dot{H}^m(\til{Q}_q)}^2\leq 2\Sob{\phi}{\dot{H}^m}^2+2\sum_{j=1}^{k-m}h^{2j}\Sob{\phi}{\dot{H}^{m+j}}^2.\notag
    \end{align}
Upon returning to the estimate of $I_{m,k}\phi$, combining the above considerations, we apply \cref{lem:multiplicity} to complete the estimate of the first sum in \eqref{est:Imk:Hl:bdd}, the fact that $h\leq 2\pi$, and Poincar\'e's inequality to finally arrive at
    \begin{align}
    \Sob{I_{m,k}\phi}{\dot{H}^\ell}^2&\leq ch^{2(1-\ell)}\Sob{\phi}{\dot{H}^k}^2+ch^{2(m-\ell)}\Sob{\phi}{\dot{H}^m}^2+ch^{2(m-\ell)}\sum_{j=1}^{k-m}h^{2j}\Sob{\phi}{\dot{H}^{m+j}}^2\notag\\
    &\leq ch^{-\ell}\Sob{\phi}{\dot{H}^k}^2,\notag
    \end{align}
as desired.
\end{proof}

\begin{Rmk}\label{rmk:const}
The universal constants appearing in the estimates in the proofs of \cref{prop:IO:global}, \cref{cor:IO:global}, and \cref{cor:IO:global:bdd} above depend additionally on $\ell,k,m$, and $\Psi$ through properties \ref{item:P1}, \ref{item:P2}, \ref{item:P4}, and \ref{item:P5}. In particular, they are always independent of the diameters associated to the covering.
\end{Rmk}

\subsection{Examples of Globalizable Local Interpolant Observable Operators}\label{sect:examples}

In this section we provide examples of local I.O.O.'s in the sense of \cref{def:IO:loc}, as well as their corresponding globalized counterparts in the sense of \eqref{eq:IO:global}. We only present a small sample of examples of immediate relevence to the context of Data Assimilation, e.g. nodal values or local averages of velocity, but remark that several other examples exist which accommodate other forms of data, e.g., nodal values or local averages of derivatives of the velocity, boundary flux data, etc. We refer the reader to \cite{ErnGuermondBook, BrennerScottBook, BrownThesis} for these additional examples. 

\subsubsection{Spectral observables} Let $Q=[a,b]^2$, where $0\leq a< b\leq 2\pi$ such that $2\pi h=b-a$, where $0<h\leq2\pi$. Then, given $N>0$, let $I^Q\phi(x)=\sum_{\bk\in\kap_h\Z^2}\chi_{B_N}(\bk)\hat{\phi}_{\bk}e^{i\bk\cdotp x}$, where $\kap_h=2\pi h^{-1}$, $\chi_{B_N}$ is the indicator function of the ball, $B_N$, of radius $N$ in $\Z^2$, centered at the origin, and $\phi_{\bk}=\kap_h^{-1}\int_Q\phi(x)e^{-i\mathbf{k}\cdotp x}dx$ denotes the Fourier coefficient of $\phi$ at wavenumber $\bk$ such that $\overline{{\phi}}_{\bk}={\phi}_{-\bk}$. Then $I^Q:{H}^{m}_{per}(Q)\goesto{H}_{per}^{k}(Q)$, for all $k\geq m\geq0$. In particular, for any $N>0$, $I^Q$ is a $Q$-local I.O.O. of all orders $m$ and at all levels $k$ with $0\leq m\leq k$ that interpolates optimally.
   
\subsubsection{Piecewise constant interpolation}\label{sect:constant} Let $Q\subset\Om$ be a bounded, open, connected subset of diameter $h>0$ such that $ch^2\leq |Q|\leq c'h^2$, for some constants $c,c'>0$. Given $\phi\in H^2(Q)$ and $x_Q\in Q$, let $T_0^{x_Q}\phi(x)=\phi(x_Q)$ define the constant function, where $\phi_Q=\phi(x_Q)$.  It was shown in \cite{AzouaniOlsonTiti2014} that $I_{0,2}^{Q}=T_0^{x_Q}$ is a $Q$--local I.O.O. of order $0$ at level $2$. In particular, one has
    \begin{align}\label{est:constant:LIO}
        \Sob{\phi-\phi(x_Q)}{L^2(Q)}^2\leq \sum_{1\leq |\al|\leq2}c_{\al}h^{2|\al|}\Sob{\bdy^\al\phi}{L^2(Q)}^2.
    \end{align}
 In light of \eqref{eq:IO:loc:approx}, we see that we may take $\eps_{0,j}(Q)^2=\sup_{|\al|=j}c_\al h^{2(|\al|+1)}$, for $j=1,2$. 
 
 Similarly, for any $x_Q\in Q$, if one defines $S_0^{x_Q}$ by $S_0^{x_Q}\phi(x)=\phi_Q$, where $\phi_Q=|Q|^{-1}\int_Q\phi(y)dy$, i.e., the so-called ``volume elements interpolant," then it was also shown in \cite{AzouaniOlsonTiti2014} that $I^Q=S_0^{x_Q}$ is a $Q$-local I.O.O. of order $0$ at level $1$ and hence, interpolates optimally. In particular,
    \begin{align}\label{est:volelt:LIO}
        \Sob{\phi-\phi_Q}{L^2(Q)}^2\leq ch^2\Sob{\nabla\phi}{L^2(Q)}^2,
    \end{align}
 for some constant $c>0$, independent of $h$; observe that $\eps_{0,1}(Q)=ch$.

\subsubsection{Taylor polynomials} Let $Q\subset\Om$ be a star-shaped bounded, open, connected subset of diameter $h>0$. Given $\phi\in H^3(Q)$ and $x_Q\in Q$ such that $|x-x_Q|\leq h$, for all $x\in Q$, let $T_1\phi(\cdotp;x_Q)$ denote the first-order Taylor polynomial of $\phi$ centered at $x_Q$. In particular
    \begin{align}\label{def:taylor:first:order}
        T_1^{x_Q}\phi(x):=T_1\phi(x;x_Q)=\phi(x_Q)+\nabla\phi(x_Q)\cdotp(x-x_Q).
    \end{align}
Then we have
    \begin{align}\label{est:taylor:LIO}
         \Sob{\phi-T_1^{x_Q}\phi}{L^2(Q)}^2\leq \sum_{2\leq|\al|\leq 3}c_{\al}h^{2|\al|}\Sob{\bdy^\al\phi}{L^2(Q)}^2.
    \end{align}
This is an elementary extension of the corresponding fact for constant interpolation proved in \cite{JonesTiti1992a, AzouaniOlsonTiti2014} in dimension $d=2$; the details are provided in \cref{sect:app:taylor}, where it is established in the greater generality of dimension $d\geq2$. Moreover, observe that
    \[
        \nabla T_1^{x_Q}\phi=\nabla\phi(x_Q)=T_0^{x_Q}\nabla\phi,
    \]
which implies
    \begin{align}\label{est:taylor:LIO:order1}
        \Sob{\phi-T_1^{x_Q}\phi}{\dot{H}^1(Q)}^2 \leq \sum_{1\leq |\al|\leq 2}c_\al' h^{2|\al|}\Sob{\bdy^\al\phi}{\dot{H}^1(Q)}^2.
    \end{align}
In particular, \eqref{est:taylor:LIO} and \eqref{est:taylor:LIO:order1} implies that \eqref{eq:IO:loc:approx} holds for $0\leq\ell\leq1$ and $k=3$, so that $I^Q=T_1^{x_Q}$ is a $Q$-local I.O.O. or order $1$ at level $3$.
Clearly $T_1^{x_Q}$ is a higher-order variant of the nodal value interpolant mentioned in the previous example. Indeed, from this point of view, $\phi(x_Q)$ simply represents the zeroth-order Taylor polynomial of $\phi$ centered at $x_Q$.

\subsubsection{Sobolev polynomials}
There are obvious shortcomings to using the Taylor polynomial as a means to interpolate nodal observations in the context of data assimilation, specifically since it requires one to make observations on derivatives of $\phi$ at given nodes. One may slightly relax this requirement by replacing nodal values of the derivatives with their spatial averages. This was done in the zeroth order case in \cref{sect:constant}, above, by replacing $\phi(x_Q)$ by $|Q|^{-1}\int_Q\phi(x)dx$ . The study of such polynomials of any order is classical and was introduced by Sobolev in \cite{SobolevMono1963}. We recall their properties here following the treatment in \cite{BrennerScottBook}. The reader is also referred to \cite{DupontScott1980}.

Let $Q\subset \Om$ be a ball of radius $h$ with center $x_Q\in\Om$. For $k\geq1$, denote the Taylor polynomial of order $k$ centered at $x_Q$ of $\phi\in C^k(\Om)$ by $T_k\phi(\cdotp;x_Q)$, so that
    \begin{align}\label{def:taylor:gen}
        T_k^{x_Q}\phi(x):=T_k\phi(x;x_Q)=\sum_{|\al|\leq k}\frac{\bdy^\al\phi(x_Q)}{\al!}(x-x_Q)^\al,
    \end{align}
where $\al\in\\pi_0^2$ is a multi-index. Fix $\til{\psi}\in C^\infty(\Om)$ to be a radial, non-negative function such that $\til{\psi}(x)=1$ when $|x|<1/2$, $\til{\psi}(x)=0$ for $|x|\geq 1$, and $\Sob{\til{\psi}}{L^1(\R^2)}=1$. Then $\{\til{\psi}_h\}_{h>0}$ is a standard mollifier, where $\til{\psi}_h(x)= h^{-2}\til{\psi}(xh^{-1})$. For $\phi\in H^{k+1}(Q)$, the corresponding $\til{\psi}_h$--averaged Taylor polynomial about $x_Q$ is then given by
    \begin{align}\label{def:taylor:gen:avg}
        S_k^{x_Q}\phi(x):=S_k\phi(x;x_Q)=\sum_{|\al|\leq k}\frac{1}{\al!}\int_{\Om}\bdy^\al\phi(y)(x-y)^\al\til{\psi}_h(y-x_Q)dy.
    \end{align}
Then (see \cite[Lemma 4.3.8]{BrennerScottBook}) for all $0\leq\ell\leq k$
    \begin{align}\label{eq:Sob:poly}
        \Sob{\phi-S_k^{x_Q}\phi}{\dot{H}^\ell(Q)}\leq c_{\ell,k}h^{k+1-\ell}\Sob{\phi}{\dot{H}^{k+1}(Q)},
    \end{align}
for some absolute constant $c_{\ell,k}$. Hence, $I^Q=S_k^{x_Q}$ is $Q$--local I.O.O. of order $k$ at level $k+1$.

\subsubsection{Lagrange polynomial}\label{ex:lagrange} 
In the context of data assimilation for the 2D NSE where it is preferable and more reasonable that velocity measurements at nodal points are collected rather than (spatial) derivatives of velocity. A class of interpolants that leverage nodal values of a function to reconstruct higher-order features of the function are Lagrange polynomials. We define them here in a configuration that fits our setting suitably, but point out that more flexibility is allowed in general, for instance, in the arrangement of the prescribed nodes. We refer the reader to \cite{BrennerScottBook} for additional details.

Let $Q\subset\Om$ be an open square such that $|Q|=h^2$. For $k\geq1$, let $\Gam_k=\{0,\dots, k\}^2$ and $\NN_Q=\{z^\gam\}_{\gam\in\Gam_k}$, where the points, $z^\gam=(z^\gam_1,z^\gam_2)$, are equally spaced nodes in $Q$. Let
    \begin{align}
        \PP_{k}&:=\{p|_{Q}: p=\sum_jp_j(x_1)q_j(x_2)\ \text{polynomial},\ \deg{p_j},\deg{q_j}\leq k\}.\label{def:poly:space}\\
        \Sigma_k &:= \{ \sigma_\gam\in (C^0(Q))':\s_\gam(f)=f(z^\gam),\ z^\gam\in\NN_{Q}\}\label{def:eval}.
    \end{align}
Note that $\Sigma_k$ represents the dual basis of $\PP_k$ and that $\dim\PP_k=|\Sigma_k|=|\NN_{Q}|=(k+1)^2$. Let $\Tht_k$ denote the basis of $\PP_k$ and represent its elements, $\tht_\gam$, by tensor products of one-dimensional polynomials as
    \begin{align}\label{def:basis}
\tht_\gam(x)=\prod_{\substack{\gam'\in\Gam_k \\ \gam_1 \neq \gam_1' \\ \gam_2\neq \gam_2'}} \frac{(x_1 - z^{\gam'}_1)( x_2 - z^{\gam'}_2)}{(z^{\gam}_1 - z^{\gam'}_1)(z^{\gam}_1 - z^{\gam'}_2)}.
    \end{align}
Finally, we define the operator $L^Q_k:H^{k+1}(Q) \to \PP_{k}$ by 
    \begin{align}\label{def:lagrange}
L^Q_k\phi(x) = \sum_{\gam \in \Gam_n}\theta_\gam(x)\sigma_\gam(\phi).
    \end{align}
Then (see \cite[Theorem 4.6.11]{BrennerScottBook}) for all $0\leq \ell \leq k'-1$ and $1\leq k'\leq k$, there exists a constant $c_{\ell,k} > 0$, independent of $\phi$ and $h$, such that
    \begin{align}\label{est:lagrange}
\Sob{\phi-L^{Q}_{k} \phi}{H^{\ell}(Q)} &\leq c_{\ell,k'} h^{k'+1-\ell} \Sob{\nabla\phi}{\dot{H}^{k'}(Q)}.
    \end{align}
Hence $I^Q=L_k^Q$ defines a $Q$-local I.O.O. of order $k'$ at level $k'+1$, for all $0\leq k'\leq k$. In fact, $L_k^Q$ interpolates optimally.

\subsubsection{Volume element polynomials}
Spatial averages of the velocity field constitute another class of physical observations. This type of data is used in the finite volume method to approximate true solutions with piecewise constant functions in the $L^2$-topology. They may also be used to construct higher-order polynomial approximations with similar error bounds to the Lagrange polynomial interpolants in higher-order Sobolev topologies.

Let $Q\subset\Om$, $\Gam_k$, $\NN_Q$, $\PP_{k}$ be as in \cref{ex:lagrange}. We define functionals given by integration on square patches within $Q$ as follows. Let
    \begin{align}\notag
        \begin{split}
\mathcal{S}_{k,Q} &:= \left\{ \mathcal{Q}^\gam = z^\gam + [0,\tfrac{h}{k}]^2 : z^\gam\in\NN_{Q}     \right\} \\
\Pi_k &:= \left\{ \pi_\gam\in
(L^1_{\text{loc}}(Q))':\pi_\gam(f)=|\mathcal{Q}^\gam|^{-1}\int_{\mathcal{Q}^\gam} f(x) dx,\ z^\gam\in\NN_{Q}\right\}.
\end{split}
    \end{align}
In \cref{sect:app:vol}, we establish that $\Pi_k$ determines basis for the dual space of $\PP_k$ and subsequently describe an explicit construction of the corresponding basis of $\PP_k$, which we denote by $\Xi_k = \{\xi_\gam\}_{\gam \in \Gam_k}$. Using this pair of bases we may define a projection operator $V^Q_k: H^{k+1}(Q)\to\PP_k$ by
\begin{align}\notag
V^Q_k\phi(x) = \sum_{\gam\in\Gam_k} \xi_\gam(x) \pi_\gam(\phi).
\end{align}
The unisolvence of the polynomial space with respect to the functionals, along with a similar argument to that for \eqref{est:lagrange} (see, for instance, \cite[Theorem 4.4.4]{BrennerScottBook}) gives the following bound: for all $0\leq \ell \leq k'-1$ and $1\leq k'\leq k$, there exists a constant $c_{\ell,k'} > 0$, independent of $\phi$ and $h$, such that
    \begin{align}\label{est:vol}
    \Sob{\phi-V^{Q}_{k} \phi}{\dot{H}^{\ell}(Q)} &\leq c_{\ell,k'} h^{k'+1-\ell} \Sob{\nabla\phi}{\dot{H}^{k'}(Q)}.
        \end{align}
Then $I^Q=V_k^Q$ defines a $Q$-local I.O.O. of order $k'$ at level $k'+1$, for all $0\leq k'\leq k$. In particular, $V_k^Q$ interpolates optimally.

\subsubsection{Hybrid interpolation}

Let $\QQ=\{Q_q\}_q$ be a covering of $\Om$ by bounded, open, connected subsets such that $h_q=\diam(Q_q)$. Given any $(\II,\Psi)_\QQ$--subordinate family, \cref{prop:IO:global} ensures that an estimate of the form \eqref{est:IO:Hl} holds. In particular, $\II$ may be any family comprising of any combination of the operators from above. Four possible categories of such combinations are given by the following.

\begin{itemize}
    \item \textit{Repeated-type, Uniform.} The $I^{(q)}$ are all of the same type of interpolating operator, e.g., all Taylor, all Sobolev, all Lagrange, etc., and there exist $m=m_q$ and $k=k_q$ for all $q$. The examples of operators considered in \cite{AzouaniOlsonTiti2014} belong strictly to this class;
    \item \textit{Repeated-type, Non-uniform.} The $I^{(q)}$ are all of the same type, but $m_q, k_q$ are allowed to vary. In this case, the induced global I.O.O. would be given by $I_{m,k}$, where $m=\inf_qm_q$ and $k=\sup_qk_q$ 
    (see \cref{rmk:IO:loc:convention});
    \item \textit{Hybrid-type, Uniform.} The $I^{(q)}$ consists of different types , but $m_q$, $k_q$  are constant in $q$;
    \item \textit{Hybrid-type, Non-uniform.} The $I^{(q)}$ consist of different types, but $m_q,k_q$ are allowed to vary in $q$.
\end{itemize}


\section{Statements of Main Results}\label{sect:statements}

We formally state our main results here. The first of the three main theorems provides estimates for the radius of the absorbing ball in $H^k$ for $k\geq 2$. In particular, we properly generalize the bounds in \cref{prop:wp:nse} to all higher orders of Sobolev regularity. Indeed, to establish the desired higher-order convergence between the nudged solution and true solution, we will make crucial use of the a priori bounds available for the true solution when initialized in the absorbing ball with respect to a Sobolev topology of arbitrary positive degree.

\begin{Thm}\label{thm:abs:ball:Hk}
Let $\gam,p\geq0$. Given $f\in \dot{H}^{k-1}_\s$, for some $k\geq2$, let $\s_k$ denote its $k$--th shape factor defined in \eqref{def:shape}. Let $u$ denote the unique, global strong solution of  \eqref{eq:nse:proj:hyper} corresponding to a bounded set of initial data $u_0\in \dot{H}_\s^k$. Then
    \begin{align}\label{eq:abs:ball:Hk}
           \frac{\Sob{u(t)}{\dot{H}^k}}{\nu}\leq&c_k\left(\s_{k-1}^{1/k}+G\right)^{k-1}G,
    \end{align}
holds for some universal constant $c_k>0$, for all $t\geq t_0$, for some $t_0$ sufficiently large, depending only on the diameter of the bounded set.
\end{Thm}

\begin{Rmk}\label{rmk:Hs}
{Note that by interpolation, one may immediately obtain absorbing ball bounds in $H^s$, for all $s>1$.}
\end{Rmk}

\noindent For the remainder of the manuscript, for each $k\geq0$, we will denote the $H^k$--absorbing ball of \eqref{eq:nse:proj:hyper} by $\B_k$ so that
    \begin{align}\label{def:abs:ball:Hk}
        \B_k=\left\{v\in\dot{H}^k_\s:\Sob{v}{\dot{H}^k}\leq c_k\left(\s_{k-1}^{1/k}+G\right)^{k-1}G\right\}.
    \end{align}

Our second theorem establishes well-posedness of the nudging-based equation in the higher-order Sobolev spaces $\dot{H}^k_\s$, where $k\geq2$, under various structural assumptions on the operator $I_{m,k}$ that detail the interplay between the system and features of the interpolation operator. 

\begin{Thm}\label{thm:wp:ng}
{Let $p,\gam\geq0$ be given such that $p=0$ if $\gam=0$}. Let $m,k$ be non-negative integers such that $1+m\leq k\leq 2+p$ and suppose that $I_{m,k}$ is an $(m,k)$--generic
$(\II,\Psi)_\QQ$--subordinate global I.O.O. Let $f\in\dot{H}^{k-1}_\s$, $u_0\in\B_m$, and $u$ be the unique, global solution of \eqref{eq:nse:proj:hyper} corresponding to $u_0,f$. 
Given any $v_0\in\dot{H}^k_{\s}$, there exists a unique $v\in C([0,T];\dot{H}^k_\s)\cap L^2(0,T;\dot{H}^{k+1})$ and $\bdy_tv\in L^2(0,T;\dot{H}^{k-1}_\s)$, which holds for all $T>0$ and satisfies \eqref{eq:nse:ng:proj:hyper} provided that 
    \begin{align}\label{cond:mu:h:B}
     c\sup_q\frac{\mu h_q^2}{\nu}\left[\veps_{0,1}(Q_q)+\veps_{0,2}(Q_q)+\chi_{(0,\infty)}(\gam)\left(\frac{\nu}{\gam}\right)\left(\frac{\mu h_q^2}{\nu}\right)\sum_{j=1}^{[p]}\veps_{0,j}(Q_q)^2h_q^{2(j-2)}\right]\leq\frac{1}{10{\pi_0}},
    \end{align}
 where $c>0$ is a universal constant, {$\pi_0$} is the constant from \ref{item:P2}, $\veps_{i,j}(Q_q)$ are the constants associated to $I_{m,k}$, and where $[p]$ denotes the greatest integer $\leq p$. {Moreover, if $I_{m,k}$ interpolates uniformly at scale $h$, then it suffices to instead assume}
   \begin{align}\label{cond:mu:h:A}
        c\frac{\mu h^2}{\nu}\left[1+\chi_{(0,\infty)}(\gam)\left(\frac{\nu}{\gam}\right)\left(\frac{\mu h^2}{\nu}\right)\sum_{j=1}^{[p]}h^{2(j-2)}\right]\leq\frac{1}{10}, 
    \end{align}
{in place of \eqref{cond:mu:h:B}}. Note that we use the convention that $\chi_{(0,\infty)}(\gam)\left(\frac{\nu}{\gam}\right)\equiv0$, when $\gam=0$.
\end{Thm}

\noindent Since the analysis performed in \cite[Theorem 6]{AzouaniOlsonTiti2014} in the periodic setting can be easily extended to prove \cref{thm:wp:ng}, we relegate the proof of this theorem to \cref{sect:app:wp:ng}.

\begin{Rmk}\label{rmk:pou}
We note that since $\Om=\T^2$ is a compact manifold without boundary, property \ref{item:P2} implies that $\QQ$ is finite. For general bounded domains, however, $\QQ$ may be infinite. We refer the reader to \cref{rmk:finite} and \cref{rmk:dirichlet} for further comments.
\end{Rmk}

The third main result provides sufficient conditions on the nudging parameter, $\mu$ and the density of data, determined by $h$, in terms of the system parameters, $\nu,f$, alone that ensure convergence of the approximating signal, as determined by the nudging-based system, to the true signal, as represented by a solution to \eqref{eq:nse:proj:hyper}, in higher-order Sobolev topologies, provided that the observables are interpolated in a suitable manner. In particular, we assume that the observables are interpolated using a sufficiently nice global I.O.O. in the sense of \eqref{eq:IO:global}.

\begin{Thm}\label{thm:sync:Hk}
{Let $p,\gam\geq0$ be given such that $p=0$ if $\gam=0$}. Let $m\geq0$, $k\geq2$ be integers such that $1+m\leq k\leq 2+p$. Let $I_{m,k}$ be an $(m,k)$--generic
$(\II,\Psi)_\QQ$--subordinate global I.O.O. Let $f\in\dot{H}^{k-1}_\s$, $u_0\in\B_m$, and $u$ be the unique, global solution of \eqref{eq:nse:proj:hyper} corresponding to $u_0,f$. Suppose that $\mu,h$ satisfy \eqref{cond:mu:wp:H1} and let $v$ denote the unique, global solution of \eqref{eq:nse:ng:proj:hyper} corresponding to $v_0\in\dot{H}^m_\s$. Then there exists a constant $0<c_0<1$ such that
    \begin{align}\label{eq:sync:gen}
        \Sob{v(t)-u(t)}{\dot{H}^{m}}\leq O(e^{-(\mu/2) t}),
    \end{align}
where $[p]$ denotes the greatest integer $\leq p$, provided that $\mu, h$ additionally satisfy
    \begin{align}\label{cond:mu:h:Hk:gen}
     c\sup_{0\leq i\leq m}\sup_q\frac{\mu h_q^2}{\nu}\left[\veps_{i,1}(Q_q)+\veps_{i,2}(Q_q)+\chi_{(0,\infty)}(\gam)\left(\frac{\mu }{\gam}\right)\sum_{j=1}^{[p]}\veps_{i,j}(Q_q)^2h_q^{2(j-1)}\right]\leq\frac{1}{10{\pi_0}},
    \end{align}
where {$\pi_0$} is the constant from \ref{item:P2}. Moreover, if $I_{m,k}$ interpolates uniformly at scale $h$, then it suffices for $\mu,h$ to instead satisfy
\begin{align}\label{cond:mu:h:Hk:unif:interpol}
        c\frac{\mu h^2}{\nu}\left[1+\chi_{(0,\infty)}(\gam)\left(\frac{\mu }{\gam}\right)\sum_{j=1}^{[p]}h^{2(j-1)}\right]\leq\frac{1}{10}, 
    \end{align}
in place of \eqref{cond:mu:h:Hk:gen}. {Note that we use the convention that $\chi_{(0,\infty)}(\gam)\left(\frac{\mu}{\gam}\right)\equiv0$, when $\gam=0$.}
\end{Thm}

Under stronger assumptions on $I_{m,k}$, one can obtain convergence up to the regularity level of the solution from which the data derives.

\begin{Thm}\label{thm:sync:Hk:mult}
{Let $\gam,p\geq0$}. Let $k\geq1$ and suppose that $I_{k+1}$ is a $(\II,\Psi)_\QQ$--subordinate global I.O.O. that interpolates optimally. Let $f\in\dot{H}^{k-1}_\s$, $u_0\in\B_k$, and let $u$ be the unique, global solution of \eqref{eq:nse:proj:hyper} corresponding to $u_0,f$. Suppose that $\mu,h$ satisfy \eqref{cond:mu:wp:H1} and let $v$ denote the unique, global solution of \eqref{eq:nse:ng:proj:hyper} corresponding to $v_0\in\dot{H}^k_\s$. Then there exists a constant $c>0$ such that
    \begin{align}\label{eq:sync:gen:mult}
        \Sob{v(t)-u(t)}{\dot{H}^k}\leq O(e^{-(\mu/2) t}),
    \end{align}
provided that $\mu, h$ additionally satisfy
    \begin{align}\label{cond:mu:h:opt:B}
        c\sum_q\veps_{k,k+1}(Q_q)^2\left(\frac{\mu h_q^2}{\nu}\right)\leq \frac{1}{10{\pi_0}}
    \end{align}
 {where {$\pi_0$} is the constant from \ref{item:P2}. Moreover, if $I_{k+1}$ is uniformly interpolating at scale $h$, then it suffices for $\mu, h$ to instead satisfy
    \begin{align}\label{cond:mu:h:opt:A}
        c\frac{\mu h^2}{\nu}\leq\frac{1}{10},
    \end{align}
in place of \eqref{cond:mu:h:opt:B}.}
\end{Thm}

\begin{Rmk}\label{rmk:gevrey}
In \cite{BiswasMartinez2017}, it was proved that convergence with respect to the analytic Gevrey norm was ensured under slightly more stringent conditions than \eqref{cond:mu:wp:H1} in the particular case when only spectral observations are used. 
\end{Rmk}

\begin{Rmk}\label{rmk:dirichlet}
In \cite{AzouaniOlsonTiti2014}, the case of Dirichlet boundary conditions was also treated, where convergence in $L^2$ was obtained for a wide class of observable quantities, including nodal value observations. In light of these results and \cref{thm:sync:Hk}, it remains an interesting issue to investigate whether one can show higher-order convergence in the setting of Dirichlet boundary conditions or others, when data is particularly given by nodal values. Moreover, in light of \cref{rmk:pou}, the framework we establish here may accommodate the case where infinitely many I.O.O.'s are used across the domain in the Dirichlet setting. We leave the study of this case to a future work.
\end{Rmk}

\begin{Rmk} From \cref{cor:IO:global}, we may immediately deduce from \cref{thm:sync:Hk}  that convergence in $H^2$ can be ensured in the case of nodal observables by using either Taylor polynomials of degree $1$ or quadratic Lagrange polynomial as the method of interpolation. Since $H^2$ embeds into $L^\infty$, this provides rigorous confirmation of the observation from the numerical simulations carried out in \cite{GeshoOlsonTiti2016} that the approximating solution was in fact converging uniformly in space to the reference solution. In particular, this also supplements the result in \cite{BiswasMartinez2017}, where the synchronization property with respect to the uniform topology, $L^\infty$,  was established in the particular case of the spectral observables. In the absence of hyperdissipation, i.e., $\gam=0$, we note that the same assumption on $\mu, h$ is imposed in either case of nodal observables or spectral observables, up to an absolute constant.
\end{Rmk}

\begin{Rmk}\label{rmk:3d}
The case of hyperviscosity is included here in order to illustrate the interplay between the order of dissipation and the order of interpolation. It is a well-known fact that for $p\geq1/4$, the corresponding system \eqref{eq:nse:proj:hyper} in dimension $d=3$ has global unique strong solutions. We point out that the analysis developed here applies in a straightforward manner to that setting as well. {We refer the reader to the work of \cite{Younsi2010} for the relevant details.}
\end{Rmk}

\section{Proofs of Main Results}\label{sect:proofs}
We will first prove \cref{thm:abs:ball:Hk} in \cref{sect:abs:ball:Hk}. We will then prove \cref{thm:sync:Hk} in \cref{sect:sync:Hk}. Recall that the proof of \cref{thm:wp:ng} will be supplied in \cref{sect:app:wp:ng}. To prove these results, it will first be useful to collect various estimates for the trilinear term that appears in the estimates. The proof of these estimates are elementary, but we supply them here for the sake of completeness.

\begin{Lem}\label{lem:nlt:gen}
Let $u$ be a smooth, divergence-free vector field in $\T^2$, and $v$ be any smooth function. Then given $\ell\geq2$, for any $|\al|=\ell$ we have
    \begin{align}\label{est:buvv:1}
        &|\lb\bdy^\al(u\cdotp\nabla )v,\bdy^\al v\rb|\\
        &\leq c\bigg(\sum_{1\leq l\leq\ell-2}\Sob{u}{\dot{H}^{\ell+1}}^{\frac{2\ell-2l-1}{2\ell}}\Sob{u}{\dot{H}^1}^{\frac{2l+1}{2\ell}}\Sob{v}{\dot{H}^{\ell+1}}^{\frac{2l+1}{2\ell}}\Sob{v}{\dot{H}^1}^{\frac{2\ell-2l-1}{2\ell}}\Sob{v}{\dot{H}^\ell}+\Sob{u}{\dot{H}^{\ell+1}}^{1/2}\Sob{u}{\dot{H}^\ell}^{1/2}\Sob{v}{\dot{H}^1}\Sob{v}{\dot{H}^{\ell+1}}^{1/2}\Sob{v}{\dot{H}^\ell}^{1/2}\notag\\
        &\qquad+\Sob{u}{\dot{H}^1}\Sob{v}{\dot{H}^{\ell+1}}\Sob{v}{\dot{H}^\ell}\bigg)\notag.
    \end{align} 
For all $\be\leq \al$, we have
    \begin{align}\label{est:bvvu}
        &|\lb(\bdy^{\al-\be}v\cdotp\nabla)\bdy^\be v,\bdy^\al u\rb|
        \leq c\Sob{u}{\dot{H}^\ell}\Sob{ v}{\dot{H}^{\ell+1}}^{\frac{2\ell-1}{\ell}}\Sob{\nabla v}{L^2}^{\frac{1}{\ell}},
    \end{align}
and
    \begin{align}\label{est:buvv}
     |\lb (\bdy^{\al-\be}u\cdotp\nabla)\bdy^\be v, \bdy^\al v\rb|\leq c\Sob{u}{\dot{H}^\ell}\left(\Sob{v}{\dot{H}^{\ell+1}}^{\frac{\ell+2}{\ell+1}}\Sob{\nabla v}{L^2}^{\frac{\ell}{\ell+1}}+\Sob{v}{\dot{H}^{\ell+1}}^{\frac{2(\ell-1)}{\ell}}\Sob{\nabla v}{L^2}^{\frac{2}{\ell}}\right),
    \end{align}
for some constants $c>0$, depending on $\ell$.
\end{Lem}

\begin{proof}
By the Leibniz rule and the fact that $u$ is divergence-free, we have
    \begin{align}
        I=\lb\bdy^\al(u\cdotp\nabla )v,\bdy^\al v\rb=\sum_{\be<\al}c_{\be,\al}\lb((\bdy^{\al-\be}u)\cdotp\nabla)\bdy^{\be}v,\bdy^\al v\rb=\sum_{\substack{\be<\al\\ 1\leq|\be|\leq \ell-2}}+\sum_{|\be|=0,\ell-1}=I_a+I_b.\notag
    \end{align}
where we interpret $\be<\al$ as $\be_i< \al_i$, for $i=1,2$.
By H\"older's inequality and interpolation we have
    \begin{align}
        |I_a|&\leq c\sum_{\substack{\be<\al\\ 1\leq|\be|\leq \ell-2}}\Sob{\bdy^{\al-\be}u}{L^4}\Sob{\nabla\bdy^\be v}{L^4}\Sob{\bdy^\al v}{L^2}
        \leq c\sum_{1\leq l\leq \ell-2}\Sob{u}{\dot{H}^{\ell-l+1}}^{1/2}\Sob{u}{\dot{H}^{\ell-l}}^{1/2}\Sob{v}{\dot{H}^{l+2}}^{1/2}\Sob{v}{\dot{H}^{l+1}}^{1/2}\Sob{\bdy^\al v}{L^2}\notag\\
        &\leq c\sum_{1\leq l\leq\ell-2}\Sob{u}{\dot{H}^{\ell+1}}^{\frac{2\ell-2l-1}{2\ell}}\Sob{u}{\dot{H}^1}^{\frac{2l+1}{2\ell}}\Sob{v}{\dot{H}^{\ell+1}}^{\frac{2l+1}{2\ell}}\Sob{v}{\dot{H}^1}^{\frac{2\ell-2l-1}{2\ell}}\Sob{\bdy^\al v}{L^2}.\notag
    \end{align}
On the other hand, by H\"older's inequality and interpolation we have
    \begin{align}
        |I_b|&\leq \Sob{\bdy^\al u}{L^4}\Sob{\nabla v}{L^2}\Sob{\bdy^\al v}{L^4}+\Sob{\nabla u}{L^2}\Sob{\bdy^\al v}{L^4}^2\notag\\
        &\leq c\Sob{u}{\dot{H}^{\ell+1}}^{1/2}\Sob{u}{\dot{H}^\ell}^{1/2}\Sob{v}{\dot{H}^1}\Sob{v}{\dot{H}^{\ell+1}}^{1/2}\Sob{v}{\dot{H}^\ell}^{1/2}+\Sob{u}{\dot{H}^1}\Sob{v}{\dot{H}^{\ell+1}}\Sob{v}{\dot{H}^\ell}\notag.
    \end{align}
Combining the estimates for $I_a$ and $I_b$, then summing over $|\al|=\ell$, yields \eqref{est:buvv:1}.

Now consider
    \begin{align}\notag
        II=\sum_{\be\leq\al}|\lb (\bdy^{\al-\be}v\cdotp\nabla)\bdy^\be u, \bdy^\al v\rb|=\sum_{\be=\al} +\sum_{|\be|=0}+\sum_{\substack{\be<\al\\ |\be|\neq0}}=II_a+II_b+II_c.
    \end{align}
We treat $II_a$ as
 \begin{align}
       |II_a|&=|\lb (v\cdotp\nabla)\bdy^\al v,\bdy^\al u\rb|\leq c\Sob{v}{L^\infty}\Sob{\nabla v}{\dot{H}^\ell}\Sob{v}{\dot{H}^\ell}\notag\\
       &\leq c\Sob{v}{\dot{H}^2}^{1/2}\Sob{v}{L^2}^{1/2}\Sob{ v}{\dot{H}^{\ell+1}}\Sob{u}{\dot{H}^\ell}
       \leq c\Sob{v}{\dot{H}^\ell}^{\frac{1}{\ell}}\Sob{v}{L^2}^{\frac{\ell-1}{\ell}}\Sob{ v}{\dot{H}^{\ell+1}}\Sob{u}{\dot{H}^\ell}.\notag
    \end{align}
We treat $II_b$ as
    \begin{align}
        |II_b|&\leq c\lb(\bdy^\al v\cdotp\nabla)u,\bdy^\al v\rb|\leq c\Sob{\bdy^\al v}{L^4}^2\Sob{\nabla u}{L^2}^2\notag\\
        &\leq c\Sob{v}{\dot{H}^{\ell+1}}\Sob{v}{\dot{H}^\ell}\Sob{\nabla u}{L^2}\leq c\Sob{v}{\dot{H}^{\ell+1}}^{\frac{\ell+1}{\ell}}\Sob{\nabla v}{L^2}^{\frac{\ell-1}{\ell}}\Sob{\nabla u}{L^2}.\notag
    \end{align}
Finally, we treat $II_c$ as
 \begin{align}
        |II_c|&\leq c\sum_{\substack{\be<\al\\ |\be|\neq0}}|\lb (\bdy^{\al-\be}v\cdotp\nabla)\bdy^\al v,\bdy^\be u\rb|\leq c\sum_{\substack{\be<\al\\ |\be|\neq0}}\Sob{\bdy^{\al-\be}v}{L^4}\Sob{\nabla v}{\dot{H}^\ell}
         \Sob{\bdy^\be u}{L^4}\notag\\
        &\leq c\Sob{ v}{\dot{H}^{\ell+1}}\sum_{\substack{\be<\al\\ |\be|\neq0}}\Sob{v}{\dot{H}^{\ell-|\be|+1}}^{1/2}\Sob{v}{\dot{H}^{\ell-|\be|}}^{1/2}\Sob{ u}{\dot{H}^{|\be|+1}}^{1/2}\Sob{ u}{\dot{H}^{|\be|}}^{1/2}\leq c\sum_{\substack{\be<\al\\ |\be|\neq0}}\Sob{ v}{\dot{H}^{\ell+1}}^{\frac{4\ell-2|\be|-1}{2\ell}}\Sob{\nabla v}{L^2}^{\frac{2|\be|+1}{2\ell}}\Sob{u}{\dot{H}^{|\be|+1}}^{1/2}\Sob{u}{\dot{H}^{|\be|}}^{1/2}.\notag
    \end{align}
Upon combining $II_a$, $II_b$, $II_c$ and applying the Poincar\'e inequality, we obtain \eqref{est:bvvu}.

Lastly, we consider
    \begin{align}
        III=\sum_{\be<\al}|\lb \bdy^{\al-\be}u^j\bdy_j\bdy^\be v^i, \bdy^\al v^i\rb|=\left(\sum_{|\be|=0}+\sum_{\substack{\be<\al\\ |\be|\neq0}}\right)=III_a+III_b.\notag
    \end{align}
We estimate $III_a$ with H\"older's inequality, interpolation, and Young's inequality to obtain
    \begin{align}
        |III_a|&\leq c\Sob{u}{\dot{H}^\ell}\Sob{\nabla v}{L^4}\Sob{\bdy^\al v}{L^4}\leq c\Sob{u}{\dot{H}^\ell}\Sob{\nabla v}{\dot{H}^1}^{1/2}\Sob{\nabla v}{L^2}^{1/2}\Sob{\nabla v}{\dot{H}^\ell}^{1/2}\Sob{v}{\dot{H}^\ell}^{1/2}\notag\\
        &\leq c\Sob{u}{\dot{H}^\ell}\Sob{\nabla v}{\dot{H}^\ell}^{\frac{\ell+2}{\ell+1}}\Sob{\nabla v}{L^2}^{\frac{\ell}{\ell+1}}.\notag
    \end{align}
We treat $III_b$ and consider the case $\ell=2$ separately. When $\ell=2$, we apply H\"older's inequality and interpolation to obtain
    \begin{align}
        |III_b|&\leq \sum_{\substack{\be<\al\\ |\be|\neq0}}\Sob{\bdy^{\al-\be}u}{L^4}\Sob{
    \bdy^\be\nabla v}{L^2}\Sob{\bdy^\al v}{L^4}\leq c\Sob{ u}{\dot{H}^{2}}^{1/2}\Sob{u}{\dot{H}^1}^{1/2}\Sob{ v}{\dot{H}^2}^{3/2}\Sob{\nabla v}{\dot{H}^2}^{1/2}\notag\\
        &\leq c\Sob{u}{\dot{H}^2}^{1/2}\Sob{u}{\dot{H}^1}^{1/2}\Sob{\nabla v}{\dot{H}^2}^{3/4}\Sob{\nabla v}{L^2}^{5/4}\notag.
    \end{align}
When $\ell\geq3$, we use the divergence-free condition to write $III_b$ as commutator. In particular
    \begin{align}\notag
        III_b=-\sum_{|\al|=\ell}\lb[\bdy^\al,u\cdotp\nabla]v,\bdy^\al v\rb.
    \end{align}
By H\"older's inequality, a classical commutator estimate \cite{KlainermanMajda1981} and interpolation we obtain
    \begin{align}
        |III_b|&\leq c\Sob{\nabla u}{L^\infty}\Sob{\nabla v}{\dot{H}^{\ell-1}}\Sob{v}{\dot{H}^\ell}+c\Sob{u}{\dot{H}^\ell}\Sob{\nabla v}{L^\infty}\Sob{v}{\dot{H}^\ell}\notag\\
        &\leq c\Sob{\nabla u}{L^\infty}\Sob{\nabla v}{\dot{H}^\ell}^{\frac{2(\ell-1)}{\ell}}\Sob{\nabla v}{L^2}^{\frac{2}{\ell}}+c\Sob{u}{\dot{H}^\ell}\Sob{\nabla v}{\dot{H}^\ell}\Sob{\nabla v}{L^2}\notag\\
        &\leq c\Sob{u}{\dot{H}^3}^{1/2}\Sob{u}{\dot{H}^1}^{1/2}\Sob{\nabla v}{\dot{H}^\ell}^{\frac{2(\ell-1)}{\ell}}\Sob{\nabla v}{L^2}^{\frac{2}{\ell}}+c\Sob{u}{\dot{H}^\ell}\Sob{\nabla v}{\dot{H}^\ell}\Sob{\nabla v}{L^2}.\notag
    \end{align}
Upon combining $III_a$, $III_b$, and applying the Poincar\'e inequality, we arrive at \eqref{est:buvv}.
\end{proof}

\subsection{Higher-order absorbing ball estimates: Proof of Theorem \ref{thm:abs:ball:Hk}}\label{sect:abs:ball:Hk}
Let $u_0\in \dot{H}^1_{\s}$ and let $u\in C([0,\infty);\dot{H}^1_{\s})\cap L^2_{loc}(0,\infty;\dot{H}^2)$ be the unique strong solution of \eqref{eq:nse:proj:hyper} corresponding to initial data $u_0$ and external forcing, $f$. Let us recall that the Grashof number is defined by \eqref{def:grashof} and the $k$--th order shape factor of $f$ is given by \eqref{def:shape}.

\begin{proof}[Proof of \cref{thm:abs:ball:Hk}]
Let $\mathcal{A}\subset\dot{H}^k_\s$ denote the global attractor of \eqref{eq:nse:proj:hyper} corresponding to forcing $f$. It will suffice to show that 
    \begin{align}\notag
        \frac{\Sob{u}{\dot{H}^k}}{\nu}\leq c(\s_{k-1}^{1/k}+G)^{k-1}G,
    \end{align}
for all $u\in\mathcal{A}$. Indeed, since $\mathcal{A}$ uniformly attracts bounded subsets of $\dot{H}^k_\s$, it would then follow that $\Sob{u(t;u_0,f)}{\dot{H}^k}\leq 2c(\s_{k-1}^{1/k}+G)^{k-1}G$, for all $t\geq t_0$, for some $t_0=t_0(u_0,f)$.

Let $u_0\in\mathcal{A}$ such that
    \begin{align}\notag
        \Sob{u_0}{\dot{H}^k}=\max\{\Sob{v}{\dot{H}^k}:v\in\mathcal{A}\}.
    \end{align} 
Denoting $u(t)=u(t;u_0,f)$, for all $t\in\R$. Then, owing to the divergence-free condition, the basic energy balance in $\dot{H}^k$ is given by
    \begin{align}
        \frac{1}2\frac{d}{dt}\Sob{u}{\dot{H}^k}^2+\nu\Sob{\nabla u}{\dot{H}^k}^2&=-\sum_{|\al|=k}\sum_{\be<\al}c_{\al,\be}\lb \bdy^{\al-\be}u^j\bdy_j\bdy^\be u^\ell,\bdy^\al u^\ell\rb+\sum_{|\al|=k}\lb \bdy^\al P_\s f,\bdy^\al u\rb\notag\\
        &=I+II.\notag
    \end{align}
In particular, by choice of $u(0)=u_0$, observe that $\frac{d}{dt}\Sob{u(t;u_0,f)}{\dot{H}^k}\big{|}_{t=0}=0$, so that
    \begin{align}
        \nu\Sob{\nabla u_0}{\dot{H}^k}^2&=I+II.\notag
    \end{align}
By \eqref{est:buvv:1} of \cref{lem:nlt:gen}, we have
    \begin{align}\label{est:nlt:decomp:Hk:I}
        |I|\leq c\Sob{\nabla u_0}{L^2}\Sob{\nabla u_0}{\dot{H}^k}\Sob{u_0}{\dot{H}^k}.
    \end{align}
For $II$, we integrate by parts and apply H\"older's inequality to obtain
    \begin{align}
        |II|&\leq\Sob{P_\s f}{\dot{H}^{k-1}}\Sob{\nabla u_0}{\dot{H}^k}=\nu^2\s_{k-1}G\Sob{\nabla u_0}{\dot{H}^k}.\notag
    \end{align}
Upon combining the estimates for $I, II$, we arrive at
    \begin{align}\notag
        \nu\Sob{\nabla u_0}{\dot{H}^k}\leq c\Sob{\nabla u_0}{L^2}\Sob{u_0}{\dot{H}^k}+\nu^2\s_{k-1}G.
    \end{align}
Observe that by interpolation
    \begin{align}\label{est:abs:Hk:interpolate}
        \Sob{u_0}{\dot{H}^k}\leq c\Sob{\nabla u_0}{\dot{H}^k}^{\frac{k-1}k}\Sob{\nabla u_0}{L^2}^{\frac{1}k}.
    \end{align}
With this and Young's inequality, we estimate
    \begin{align}
        \nu\Sob{\nabla u_0}{\dot{H}^k}&\leq c\Sob{\nabla u_0}{\dot{H}^k}^{\frac{k-1}k}\Sob{\nabla u_0}{L^2}^{\frac{k+1}k}+\nu^2\s_{k-1}G\notag\\
        &\leq \frac{\nu}2\Sob{\nabla u_0}{\dot{H}^k}+c\nu^3\left(\frac{\Sob{\nabla u_0}{L^2}}{\nu}\right)^{k+1}+\nu^2\s_{k-1}G\notag\\
        &\leq \frac{\nu}2\Sob{\nabla u_0}{\dot{H}^k}+c\nu^2G^{k+1}+\nu^2\s_{k-1}G.\notag
    \end{align}
Hence
    \begin{align}\label{est:abs:gradHk:max}
        \Sob{\nabla u_0}{\dot{H}^k}\leq c\nu(\s_{k-1}^{1/k}+G)^kG.
    \end{align}
Returning to \eqref{est:abs:Hk:interpolate}, it follows that
    \begin{align}\notag
        \frac{\Sob{u_0}{\dot{H}^k}}{\nu}\leq c(\s_{k-1}^{1/k}+G)^{k-1}G^{\frac{k-1}k}\left(\frac{\Sob{\nabla u_0}{L^2}}{\nu}\right)^{\frac{1}k}\leq c(\s_{k-1}^{1/k}+G)^{k-1}G,
    \end{align}
which proves \eqref{eq:abs:ball:Hk}, as desired.

\end{proof}

\subsection{Synchronization in higher-order Sobolev topologies}\label{sect:sync:Hk}
Let $p\geq0$ and $u_0,v_0\in \B_1\cap\B_k$, where $1+m\leq k\leq 2+p$ and $u$ and $v$ denote the corresponding unique strong solutions of the following initial value problem 
\begin{align}\label{eq:nudge:newton}
\begin{split}
\bdy_tu -\nu\De u+\gam(-\De)^{p+1}u+ P_\s(u \cdot \nabla ) u  &=P_\s f, \quad P_\s u=0,\quad u(0) = u_0, \\
\bdy_tv -\nu\De v+\gam(-\De)^{p+1}v+ P_\s(v \cdot \nabla ) v  &= P_\s f - \mu P_\s J_{m,k}(v - u),\quad P_\s v=0,\quad v(0)=v_0,
\end{split}
\end{align}
where $J_{m,k}=I_{m,k}-\lb I_{m,k}\rb$, where $I_{m,k}$ is an $(\II,\Psi)_\QQ$--subordinate global I.O.O. with associated covering $\QQ$, and $\lb I_{m,k}\rb$ denotes the operator such that $\lb I_{m,k}\rb\phi=(2\pi)^{-2}\int_{\T^2} I_{m,k}\phi(x)dx$. Let $w:=v-u$ and $w_0=v_0-u_0$, so that $w$ satisfies
    \begin{align}\label{eq:nudge:diff}
    \bdy_tw-\nu\De w+\gam(-\De)^{p+1}w+P_\s(w\cdotp\nabla)w+P_\s(w\cdotp\nabla)u+P_\s(u\cdotp\nabla) w=-\mu P_\s J_{m,k}w,\quad P_\s w=0,\quad w(0)=w_0.
    \end{align}
We will ultimately show that $\lim_{t\goesto\infty}\Sob{w(t)}{\dot{H}^\ell}=0$. Our approach will be to bootstrap convergence in higher-order Sobolev topologies, starting from $H^1$. {Note that we adopt the following convention for applications of the Cauchy-Schwarz inequality or Young's inequality in the analysis below, which we will invoke repeatedly:
	\begin{align}\notag
		ab\leq ca^p+\frac{1}{100}b^{p'},\quad a,b\geq0,\quad \frac{1}p+\frac{1}{p'}=1,
	\end{align}
for some $c>0$ depending on $p,p'$. There is nothing essential about the constant $1/100$, except that we never add more than $50$ of such terms in a given argument. In particular, we make no attempt whatsoever to optimize such constants. This can certainly be done by the interested reader, but in order for this to be a meaningful exercise, one must also carefully track the constants from \cref{lem:nlt:gen}, which we also neglect to do. The important feature that we care to emphasize is the manner in which the constants from the I.O.O.'s appear in the analysis, as the development of these operators is the main novelty of this work.}

\begin{Lem}\label{lem:sync:H1}
Let $m\geq0$ and $1+m\leq k\leq 2+p$. Let $I_{m,k}$ be an  $(\II,\Psi)_\QQ$--subordinate global I.O.O. that is $(m,k)$--generic. Suppose that $\mu,h$ satisfies
    \begin{align}\label{cond:mu:h:H1:B}
   c\sup_q\left(\frac{\mu h_q^2}{\nu}\right)\left[\veps_{0,1}(Q_q)^2+\veps_{0,2}(Q_q)^2+\chi_{(0,\infty)}(\gam)\left(\frac{\mu h_q^2}{\gam}\right)\sum_{j=1}^{[p]}\veps_{0,j+2}(Q_q)^2h_q^{2j}\right]\leq\frac{1}{10\pi_0},
    \end{align}
where $[p]$ denotes the greatest integer $\leq p$. Then for $c,c'>0$ sufficiently large and $\mu$ additionally satisfying
    \begin{align}\label{cond:mu:AOT}
        \mu\geq c'\nu(1+\log(1+G))G,
    \end{align}
one has
    \begin{align}\label{est:H1:w}
     \Sob{\nabla w(t)}{L^2}^2\leq e^{-\frac{3}2\mu t}\Sob{\nabla w_0}{L^2}^2\quad\text{and}\quad \int_0^t\Sob{\nabla w(s)}{L^2}^2ds\leq \frac{\Sob{\nabla w_0}{L^2}^2}{\mu}.
    \end{align}
If, additionally, $I_{m,k}=I_{k+1}$ interpolates optimally, then it suffices for $\mu, h$ to satisfy
    \begin{align}\label{cond:mu:h:H1:opt:B}
        c\sup_q\veps_{1,1}(Q_q)^2\left(\frac{\mu h_q^2}{\nu}\right)\leq\frac{1}{10{\pi_0}},
    \end{align}
where ${\pi_0}$ is the constant from \ref{item:P2}.

On the other hand, if $I_{m,k}$ is uniformly interpolating at scale $h$, then we may instead suppose that $\mu, h$ satisfies   
    \begin{align}\label{cond:mu:h:H1:A}
        c\frac{\mu h^2}{\nu}\left(1+\chi_{(0,\infty)}(\gam)\sum_{j=1}^{[p]}h^{2j}\right)\leq \frac{1}{10},
    \end{align}
in place of \eqref{cond:mu:h:H1:B}. If, additionally, $I_{m,k}=I_{k+1}$ interpolates optimally, then it suffices for $\mu, h$ to satisfy
    \begin{align}\label{cond:mu:h:H1:opt:A}
        c\frac{\mu h^2}{\nu}\leq \frac{1}{10}.
    \end{align}
\end{Lem}

\begin{proof}
Upon taking the $L^2$--inner product of \eqref{eq:nudge:diff} with $-P_\s\De w$, integrating by parts, then using the fact that $\lb (w\cdotp\nabla) w,\De w\rb=0$, we obtain
    \begin{align}\notag
        \frac{1}2\frac{d}{dt}\Sob{\nabla w}{L^2}^2&+\nu\Sob{\De w}{L^2}^2+\mu\Sob{\nabla w}{L^2}^2+\gam\Sob{(-\De)^{1+p/2}w}{L^2}^2\notag\\
        &=\lb w\cdotp\nabla u+u\cdotp\nabla w,\De w\rb-\mu\lb w-J_{m,k}w,\De w\rb=I+II.\notag
    \end{align}
To estimate $I$, we again invoke the property that $\lb w\cdotp\nabla w,\De w\rb=0$, then apply H\"older's inequality, the Br\'ezis-Gallouet inequality, and \cref{prop:wp:nse} to obtain
    \begin{align}
        |I|=|\lb w\cdotp\nabla w,\De u\rb|\leq \Sob{w}{L^\infty}\Sob{\nabla w}{L^2}\Sob{\De u}{L^2}\leq c\nu\Sob{\nabla w}{L^2}^2\left(1+\log\frac{\Sob{\De w}{L^2}^2}{\Sob{\nabla w}{L^2}^2}\right)(\s_1^{1/2}+G)G.\notag
    \end{align}
    
To estimate $II$, we appeal to \cref{prop:IO:global} and invoke \eqref{cond:mu:h:H1:B} and Poincar\'e's inequality, so that
    \begin{align}
    |II_A|
    &\leq c\frac{\mu^2}{\nu}\sum_{j=1}^{k}\sum_q\veps_{0,j}(Q_q)^2h_q^{2j}\Sob{w}{\dot{H}^j(\til{Q}_q)}^2+\frac{\nu}{100}\Sob{\De w}{L^2}^2\notag\\
    &=c\mu\sum_q\left[\veps_{0,1}(Q_q)^2\left(\frac{\mu h_q^2}{\nu}\right)-\frac{1}{10{\pi_0}}\right]\Sob{\nabla w}{L^2(\til{Q}_q)}^2+c\nu\sum_q\left[\veps_{0,2}(Q_q)^2\left(\frac{\mu h_q^2}{\nu}\right)^2-\frac{1}{10\pi_0}\right]\Sob{\De w}{L^2(\til{Q}_q)}^2\notag\\
    &\quad+c\gam\sum_q\sum_{j=1}^{[p]}\left[\veps_{0,j+2}(Q_q)^2\left(\frac{\mu h_q^{2}}{\gam}\right)^2h_q^{2j}-\frac{1}{10{\pi_0}}\right]\Sob{\De w}{\dot{H}^j(\til{Q}_q)}^2\notag\\
    &\quad+\frac{\mu}{10}\Sob{\nabla w}{L^2}^2+\frac{\nu}{10}\Sob{\De w}{L^2}^2+\frac{\gam}{10}\Sob{(-\De)^{1+p/2}w}{L^2}^2\notag\\
    &\leq\frac{\mu}{10}\Sob{\nabla w}{L^2}^2+\frac{\nu}{10}\Sob{\De w}{L^2}^2+\frac{\gam}{10}\Sob{(-\De)^{1+p/2}w}{L^2}^2\notag.
    \end{align}
If $I_{m,k}=I_{k+1}$ is further assumed to interpolate optimally, we then proceed to apply Young's inequality, \eqref{est:IO:optimal} of \cref{prop:IO:global}, and invoke \eqref{cond:mu:h:H1:opt:B}, so that
    \begin{align}
        |II_A'|&\leq c\mu\sum_q\veps_{1,1}(Q_q)^2h_q^2\Sob{\De w}{L^2(\til{Q}_q)}^2+\frac{\mu}{100}\Sob{\nabla w}{L^2}^2\notag\\
        &\leq c\nu\sum_q\left[\veps_{1,1}(Q_q)^2\left(\frac{\mu h_q^{2}}{\nu}\right)-\frac{1}{10{\pi_0}}\right]\Sob{\De w}{L^2(\til{Q}_q)}^2+\frac{\nu}{10}\Sob{\De w}{L^2}^2+\frac{\mu}{100}\Sob{\nabla w}{L^2}^2\notag\\
        &\leq \frac{\nu}{10}\Sob{\De w}{L^2}^2+\frac{\mu}{100}\Sob{\nabla w}{L^2}^2.\notag
    \end{align}

Now let us consider the case when $I_{m,k}$ interpolates uniformly at scale $h$. To estimate $II$, we apply the Cauchy-Schwarz inequality, \cref{cor:IO:global}, Young's inequality, and \eqref{cond:mu:h:H1:A} to obtain
    \begin{align}
        |II_B|&\leq \mu\Sob{w-I_{m,k}w}{L^2}\Sob{\De w}{L^2}\leq c\frac{\mu^2}{\nu}\sum_{j=1}^{k}h^{2j}\Sob{w}{\dot{H}^j}^2+\frac{\nu}{100}\Sob{\De w}{L^2}^2\notag\notag\\
        &= c\mu\left(\frac{\mu h^2}{\nu}\right)\Sob{\nabla w}{L^2}^2+c\nu\left(\frac{\mu h^2}{\nu}\right)^2\Sob{\De w}{L^2}^2+c\gam\left(\frac{\nu}{\gam}\right)\left(\frac{\mu h^2}{\nu}\right)^2\sum_{j=1}^{[p]}h^{2j}\Sob{\De w}{\dot{H}^{j}}^2+\frac{\nu}{100}\Sob{\De w}{L^2}^2\notag\\
        &\leq\frac{\mu}{10}\Sob{\nabla w}{L^2}^2+\frac{\nu}{10}\Sob{\De w}{L^2}^2+\frac{\gam}{10}\Sob{(-\De)^{1+p/2}\De w}{L^2}^2.\notag
    \end{align}
If $I_{m,k}=I_{k+1}$ is further assumed to interpolate optimally, then we integrate by parts first to write
    \begin{align}
        II'=\mu \lb\nabla(w-I_{k+1}w),\nabla w\rb,\notag
    \end{align}
where we have used the fact that $\nabla J_{k+1}=\nabla I_{k+1}$. If $I_{k+1}$ also interpolates uniformly at scale $h$, then by \eqref{est:IO:regularly:optimally} of \cref{cor:IO:global}, Young's inequality, and \eqref{cond:mu:h:H1:opt:A}, it follows that
    \begin{align}
        |II'_B|\leq \mu\Sob{\nabla(w-I_{k+1}w)}{L^2}\Sob{\nabla w}{L^2}\leq \mu ch\Sob{\De w}{L^2}\Sob{\nabla w}{L^2}\leq\frac{\nu}{10}\Sob{\De w}{L^2}^2+\frac{\mu}{10}\Sob{\nabla w}{L^2}^2.\notag
    \end{align}

Upon combining $I$ and either $II_A$ or $II_A'$ or $II_B$ or $II_B'$, we have
    \begin{align}
        \frac{d}{dt}\Sob{\nabla w}{L^2}^2+\frac{3}2\gam\Sob{(-\De)^{1+p/2} w}{L^2}^2\leq -\mu\left\{\frac{8}5+\frac{\nu}{\mu}\left[\frac{\Sob{\De w}{L^2}^2}{\Sob{\nabla w}{L^2}^2}-(\s_1^{1/2}+G)G\left(1+\log\frac{\Sob{\De w}{L^2}^2}{\Sob{\nabla w}{L^2}^2}\right)\right]\right\}\Sob{\nabla w}{L^2}^2.\notag
    \end{align}
Since $\mu$ additionally satisfies \eqref{cond:mu:AOT} (see \cite[Lemma 2]{AzouaniOlsonTiti2014}), it follows that
    \begin{align}\notag
        \frac{d}{dt}\Sob{\nabla w}{L^2}^2\leq -\frac{3}2\mu\Sob{\nabla w}{L^2}^2,
    \end{align}
so that Gronwall's inequality yields
    \begin{align}\notag
        \Sob{\nabla w(t)}{L^2}^2\leq e^{-\frac{3}2\mu t}\Sob{\nabla w_0}{L^2}^2,
    \end{align}
as desired.
\end{proof}

We are now ready to prove \cref{thm:sync:Hk}.

\begin{proof}[Proof of \cref{thm:sync:Hk}]
Observe that \cref{lem:sync:H1} covers the case $m=0,1$. It suffices then to consider $2\leq\ell\leq m$. Let $\bdy^\al$ denote any partial with respect to $x$ of order $|\al|=\ell$. Upon applying $\bdy^\al$ to \eqref{eq:nudge:diff}, taking the $L^2$--inner product of the result with $\bdy^\al w$, integrating by parts, then summing over all $|\al|=\ell$, we obtain the following energy balance
    \begin{align}\label{eq:nudge:diff:Hk}
        \frac{1}2&\frac{d}{dt}\Sob{w}{\dot{H}^\ell}^2+\nu\Sob{\nabla w}{\dot{H}^\ell}^2+\mu\Sob{w}{\dot{H}^\ell}^2+\gam\Sob{(-\De)^{p/2}\nabla w}{\dot{H}^\ell}^2\notag\\
        =&-\sum_{|\al|=\ell}\left(\lb \bdy^\al(w\cdotp\nabla)w+\bdy^\al(u\cdotp\nabla)w+\bdy^\al(w\cdotp\nabla)u,\bdy^\al w\rb\right)+\mu\sum_{|\al|=\ell}\lb\bdy^\al(w-J_{m,k}w),\bdy^\al w\rb=I+II
    \end{align}
We estimate each the terms $I$ and $II$. Observe that $I$ may be expanded as
    \begin{align}
    I&=-\lb \bdy^{\al}(w\cdotp\nabla)w, \bdy^\al w\rb-\sum_{|\al|=\ell}\sum_{\be\leq\al}c_{\al,\be}\lb \bdy^{\al-\be}w^j\bdy_j\bdy^\be u^i, \bdy^\al w^i\rb-\sum_{|\al|=\ell}\sum_{\be<\al}c_{\al,\be}\lb \bdy^{\al-\be}u^j\bdy_j\bdy^\be w^i, \bdy^\al w^i\rb\notag\\
    &=I_a+I_b+I_c.\notag    
    \end{align}
We treat $I_a$ with \eqref{est:buvv:1} of \cref{lem:nlt:gen}, then apply Young's inequality to obtain
    \begin{align}\notag
        |I_a|&\leq c\Sob{\nabla w}{L^2}\Sob{\nabla w}{\dot{H}^\ell}\Sob{w}{\dot{H}^\ell}\leq \frac{\nu}{100}\Sob{\nabla w}{\dot{H}^\ell}^2+\frac{c}{\nu}\Sob{\nabla w}{L^2}^2\Sob{w}{\dot{H}^\ell}^2.
    \end{align}
We treat $I_b$ by first integrating by parts, then applying \eqref{est:bvvu} of \cref{lem:nlt:gen} and Young's inequality to obtain 
    \begin{align}
       |I_b|
        &\leq c\Sob{u}{\dot{H}^\ell}\Sob{ w}{\dot{H}^{\ell+1}}^{\frac{2\ell-1}{\ell}}\Sob{\nabla w}{L^2}^{\frac{1}{\ell}}\leq \frac{\nu}{100}\Sob{\nabla w}{\dot{H}^{\ell}}^2+c\nu\left(\frac{\Sob{u}{\dot{H}^\ell}}{\nu}\right)^{2\ell}\Sob{w}{\dot{H}^1}^2.\notag
    \end{align}
For $I_c$, we apply \eqref{est:buvv} of \cref{lem:nlt:gen} and Young's inequality to obtain
    \begin{align}
        |I_c|&\leq c\Sob{u}{\dot{H}^\ell}\left(\Sob{\nabla w}{\dot{H}^{\ell}}^{\frac{\ell+2}{\ell+1}}\Sob{\nabla w}{L^2}^{\frac{\ell}{\ell+1}}+\Sob{\nabla w}{\dot{H}^{\ell}}^{\frac{2(\ell-1)}{\ell}}\Sob{\nabla w}{L^2}^{\frac{2}{\ell}}\right)\notag\\
        &\leq c\nu\left[\left(\frac{\Sob{u}{\dot{H}^\ell}}{\nu}\right)^{\frac{2(\ell+1)}{\ell}}+\left(\frac{\Sob{u}{\dot{H}^\ell}}{\nu}\right)^{\frac{1}\ell}\right]\Sob{w}{\dot{H}^1}^2+\frac{\nu}{100}\Sob{\nabla w}{\dot{H}^\ell}^2.\notag
    \end{align}
Upon combining the estimates $I_a$, $I_b$, $I_c$, then applying \cref{thm:abs:ball:Hk}, we obtain
    \begin{align}\notag
        |I|\leq \frac{\nu}{10}\Sob{\nabla w}{\dot{H}^\ell}^2+\frac{c}\nu\Sob{\nabla w}{L^2}^2\Sob{w}{\dot{H}^\ell}^2+c\nu\left(1+\left(\s_{\ell-1}^{1/\ell}+G\right)^{\ell-1}G\right)^{2\ell}\Sob{w}{\dot{H}^1}^2
    \end{align}
    
Now we treat $II$. First observe that upon integrating by parts, we get
    \begin{align}\label{est:II:approach}
        II=-\mu\sum_{|\al|=\ell}\sum_{\substack{\al'<\al\\ |\al'|=1}}\lb \bdy^{\al-\al'}(w-I_{m,k}w),\bdy^{\al+\al'}w\rb.
    \end{align}
Note that we used the fact that $P_\s$ commute with derivatives and that $\bdy^{\be}J_{m,k}=\bdy^{\be}I_{m,k}$, for any $|\be|>0$. then by the Cauchy-Schwarz inequality, \cref{prop:IO:global}, Young's inequality, and the assumptions that $1\leq k\leq 2+p$, we have 
    \begin{align}\label{est:diff:HK:IV:gen}
    |II|\leq&c\frac{\mu^2}{\nu}\sum_{j=1}^{k}\sum_q\veps_{\ell-1,j}(Q_q)^2h_q^{2(j-\ell+1)}\Sob{w}{\dot{H}^{j}(\til{Q}_q)}^2+\frac{\nu}{100}\Sob{\nabla w}{\dot{H}^\ell}^2\notag\\
   \leq& c\frac{\mu^2}{\nu}\sum_{j=1}^{\ell-1}\sum_q\veps_{\ell-1,j}(Q_q)^2h_q^{2(j-\ell+1)}\Sob{w}{\dot{H}^{j}(\til{Q}_q)}^2+\frac{\nu}{100}\Sob{\nabla w}{\dot{H}^\ell}^2\notag\\
        &+c\mu\sum_q\veps_{\ell-1,\ell}(Q_q)^2\left(\frac{\mu h_q^2}{\nu}\right)\Sob{w}{\dot{H}^\ell(\til{Q}_q)}^2+c\nu\sum_q\veps_{\ell-1,\ell+1}(Q_q)^2\left(\frac{\mu h_q^2}{\nu}\right)^2\Sob{\nabla w}{\dot{H}^\ell(\til{Q}_q)}^2\notag\\
        &+c\gam\left(\frac{\nu}{\gam}\right)\chi_{[2,\infty)}(\ell)\sum_q\left(\frac{\mu h_q^2}\nu\right)^2\sum_{j=1}^{k-\ell-1}\veps_{\ell-1,j+\ell+1}(Q_q)^2h_q^{2j}\Sob{\nabla w}{\dot{H}^{\ell+j}(\til{Q}_q)}^2.
    \end{align}
On the other hand, if $I_{m,k}$ is uniformly interpolating at scale $h$, we estimate as above and apply \cref{cor:IO:global} in place of \cref{prop:IO:global} to obtain
    \begin{align}
        |II'|
        \leq& c\frac{\mu^2}{\nu}\sum_{j=1}^{k}h^{2(j-\ell+1)}\Sob{w}{\dot{H}^{j}}^2+\frac{\nu}{100}\Sob{\nabla w}{\dot{H}^\ell}^2\notag\\
        &+c\mu\left(\frac{\mu h^2}{\nu}\right)\Sob{w}{\dot{H}^\ell}^2+c\nu\left(\frac{\mu h^2}{\nu}\right)^2\Sob{\nabla w}{\dot{H}^\ell}^2+c\gam\left(\frac{\nu}{\gam}\right)\left(\frac{\mu h^2}\nu\right)^2\left(\sum_{j=1}^{[p]}h^{2j}\right)\Sob{(-\De)^{p/2}\nabla w}{\dot{H}^{\ell}}^2.\label{est:diff:Hk:IV:unif:interpol}
    \end{align}
    
Finally, upon returning to \eqref{eq:nudge:diff:Hk} and combining the estimates for $I$ and either $II$ or $II'$, then applying the Poincar\'e inequality and \eqref{cond:mu:h:B} or \eqref{cond:mu:h:A}, respectively, we see that     
    \begin{align}\label{est:sync:Hk}
      \frac{d}{dt}&\Sob{w}{\dot{H}^\ell}^2+\frac{9}5\nu\Sob{\nabla w}{\dot{H}^\ell}^2+\mu\left(\frac{9}5-c\frac{\nu}{\mu}\frac{\Sob{\nabla w}{L^2}^2}{\nu^2}\right)\Sob{w}{\dot{H}^\ell}^2+\frac{3}2\gam\Sob{(-\De)^{p/2}\nabla w}{\dot{H}^\ell}^2\\
        \leq&c\nu\left(1+\left(\s_{\ell-1}^{1/\ell}+G\right)^{\ell-1}G\right)^{2\ell}\Sob{\nabla w}{L^2}^2+c\nu\left(\frac{\mu}{\nu}\right)^2(\inf_qh_q)^{2-\ell}\Sob{w}{\dot{H}^{\ell-1}}^2.\notag
    \end{align}   
Lastly, we invoke the fact that $u_0\in\B_\ell$, the estimate \eqref{est:H1:w} of \cref{lem:sync:H1}, and ultimately Gronwall's inequality to deduce 
    \begin{align}\label{est:Hk:sync:final}
        \Sob{w(t)}{\dot{H}^\ell}^2\leq& e^{-\frac{3}2\mu t}\left[\Sob{w_0}{\dot{H}^\ell}^2+c\frac{\nu}{\mu}\left(1+\left(\s_{\ell-1}^{1/\ell}+G\right)^{\ell-1}G\right)^{2\ell}\Sob{\nabla w_0}{L^2}^2\right]e^{c\frac{\Sob{\nabla w_0}{L^2}^2}{\mu^2}}\notag\\
        &+c\nu\left(\frac{\mu}{\nu}\right)^2h^{2-\ell} \left(\int_0^te^{-\mu(t-s)}\Sob{w(s)}{\dot{H}^{\ell-1}}^2ds\right)e^{c\frac{\Sob{\nabla w_0}{L^2}^2}{\mu^2}},
    \end{align}
for all $2\leq\ell \leq m$. The remainder of the proof can be completed by a basic induction argument, where the assumption of the induction step is $\Sob{w(s)}{\dot{H}^{\ell-1}}^2\leq O(e^{-\mu s})$.
\end{proof}

\begin{proof}[Proof of \cref{thm:sync:Hk:mult}]
We let $\ell=k$ and proceed exactly as in proof of \cref{thm:sync:Hk}, except that we instead treat $II$ without integrating by parts in \eqref{est:II:approach} and use \eqref{est:IO:optimal} of \cref{prop:IO:global} to obtain 
    \begin{align}\notag
        |II|&\leq c\mu \Sob{w-I_{k+1}w}{\dot{H}^k}\Sob{w}{\dot{H}^k}\leq c\mu\Sob{w-I_{k+1}}{\dot{H}^k}^2+\frac{\mu}{100}\Sob{w}{\dot{H}^k}^2\notag\\
        &\leq c\nu\sum_q\veps_{k,k+1}(Q_q)^2\left(\frac{\mu h_q^2}{\nu}\right)\Sob{\nabla w}{\dot{H}^{k}(Q_q)}^2+\frac{\mu}{100}\Sob{w}{\dot{H}^k}^2.\notag
    \end{align}
Now observe that from \cref{lem:multiplicity}, we have
    \begin{align}
        \nu\Sob{\nabla w}{\dot{H}^k}^2\geq \frac{\nu}{{\pi_0}}\sum_q\Sob{\nabla w}{\dot{H}^k(Q_q)}^2.\notag
    \end{align}
Hence, arguing as we did in \cref{thm:sync:Hk}, we arrive at
   \begin{align}
      &\frac{d}{dt}\Sob{w}{\dot{H}^\ell}^2+\nu\sum_q\left(\frac{9}{5{\pi_0}}-\veps_{k,k+1}(Q_q)^2\left(\frac{\mu h_q^2}{\nu}\right)\right)\Sob{\nabla w}{\dot{H}^\ell(Q_q)}^2+\mu\left(\frac{9}5-\frac{\nu}{\mu}\frac{\Sob{\nabla w}{L^2}^2}{\nu^2}\right)\Sob{w}{\dot{H}^\ell}^2\notag\\
      &\leq c\nu\left(1+\left(\s_{\ell-1}^{1/\ell}+G\right)^{\ell-1}G\right)^{2\ell}\Sob{\nabla w}{L^2}^2.\notag
    \end{align}
Upon invoking \eqref{cond:mu:h:opt:B}, then applying of Gronwall's inequality and \cref{lem:sync:H1}, we obtain \eqref{eq:sync:gen:mult}.

On the other hand, if $I_{k+1}$ is also uniformly interpolating at scale $h$, then
    \begin{align}
        |II'|&\leq  c\nu\left(\frac{\mu h^2}{\nu}\right)\Sob{\nabla w}{\dot{H}^{k}}^2+\frac{\mu}{100}\Sob{w}{\dot{H}^k}^2.\notag
    \end{align}
Arguing we did before, we obtain
    \begin{align}
      \frac{d}{dt}\Sob{w}{\dot{H}^\ell}^2+\nu\left(\frac{9}{5}-c\left(\frac{\mu h^2}{\nu}\right)\right)\Sob{\nabla w}{\dot{H}^\ell}^2+\mu\left(\frac{9}5-\frac{\nu}{\mu}\frac{\Sob{\nabla w}{L^2}^2}{\nu^2}\right)\Sob{w}{\dot{H}^\ell}^2\leq c\nu\left(1+\frac{\Sob{u}{\dot{H}^\ell}}{\nu}\right)^{2\ell}\Sob{\nabla w}{L^2}^2,\notag
    \end{align}
We then apply \eqref{cond:mu:h:opt:A} and deduce \eqref{eq:sync:gen:mult} once again.

\end{proof}

\subsection*{Acknowledgements}
The work
of Vincent R. Martinez was partially supported by the award PSC-CUNY 64335-00 52, jointly funded by The Professional Staff Congress and The City University of New York. The authors would like to thank Michael S. Jolly and Ali Pakzad for insightful discussions in the course of this work, as well as the referees for their careful reading of the manuscript and the generous comments they shared to improve it.
\appendix

\section{Well-posedness of nudging-based equation in higher-order Sobolev spaces}\label{sect:app:wp:ng}

We will now supply the proof of \cref{thm:wp:ng} in $\dot{H}^k_\s$, where $k\geq2$. Recall that we will consider the system \eqref{eq:nse:proj:hyper}. Recall that for $p>0$, we denote by $(-\De)^p$ the operator defined by $\widehat{(-\De)^p\phi}(k)=|k|^{2p}\hat{\phi}(k)$, whenever $k\notin\Z^2\smod\{0\}$. Given $\gam\geq0$, consider
    \begin{align}\label{eq:nse:proj:hyper:app}
        \bdy_tu-\nu\De u+\gam(-\De)^{p+1}u+P_\s(u\cdotp\nabla)u=P_\s f,\quad P_\s u=0,\quad u(0)=u_0,
    \end{align}
where $u_0\in\dot{H}^k_\s$ and $f\in \dot{H}^{k-1}_\s$, where $k\geq2$. We note that the analog of \cref{prop:wp:nse} still holds for \eqref{eq:nse:proj:hyper:app}, for all $p>0$, as well as all bounds for the $\gam=0$ case. For $m\geq0$ such that $k\geq m+1$, let $I_{m,k}$ an $(\II,\Psi)_\QQ$--subordinate global I.O.O., let
    \begin{align}\label{def:fmu}
        f_\mu := P_\s f + \mu P_\s J_{m,k} u,
    \end{align}
where $J_{m,k}:=I_{m,k}-\lb I_{m,k}\rb$, $\lb I_{m,k}\rb$ denotes the operator such that $\lb I_{m,k}\rb\phi=(2\pi)^{-2}\int_{\T^2}I_{m,k}\phi(x)dx$, and $P_\s$ denotes the Leray projector (see \eqref{def:leray}). Let $u\in C([0,\infty);\dot{H}^k_\s)\cap L^2_{loc}(0,\infty;\dot{H}^{k+1})$ denote the unique, global solution of \eqref{eq:nse:proj:hyper:app} corresponding to $u_0\in\dot{H}^k_\s$ and $f$ guaranteed by \cref{prop:wp:nse}. Then for $v_0\in\dot{H}^k_\s$, we consider the initial value problem
    \begin{align}\label{eq:nse:ng:proj:hyper:app}
        \bdy_tv-\nu\De v+\gam(-\De)^{p+1}v+P_\s(v\cdotp\nabla v)v=P_\s f-\mu P_\s J_{m,k}(v-u),\quad P_\s v=0,\quad v(0)=v_0.
    \end{align}

We prove the existence of solutions to \eqref{eq:nse:ng:proj:hyper} via Galerkin approximation. Let $P_N$ denote the Galerkin projection at level $N>0$ and let $v_N$ denote the unique solution to the following system of ODEs
    \begin{align}\label{eq:nse:ng:gal}
        \begin{split}
        &\frac{dv_N}{dt} -\nu\De v_N +\gam(-\De)^{p+1} v_N+ P_N P_\s(v_N\cdotp\nabla)v_N = P_Nf_\mu- \mu P_N J_{m,k}v_N,\quad\nabla\cdotp v_N=0,\quad v_N(0,x) = P_N v_0(x),
        \end{split}
    \end{align}
where $p\geq0$ and $\gam\geq0$. We will develop uniform bounds for $v_N$ in the appropriate topology over the maximal interval of existence $[0,T_N)$, independent of $N$. This will imply global existence for the projected system, and therefore global existence for \eqref{eq:nse:ng:proj:hyper:app}. Uniqueness and continuity with respect to initial data will follow along the same lines as in \cite{AzouaniOlsonTiti2014}.

To prove \cref{thm:wp:ng}, we will make use of the following lemma, which controls the growth of $J_{m,k}u$, where $u$ is a strong solution to \eqref{eq:nse:proj:hyper} evolving within an absorbing ball.

\begin{Lem}\label{lem:Jmk:Hl}
Let $m\geq0$ and $k\geq2$ such that $k\geq m+1$. Given $f\in \dot{H}^{k-1}$, let $\B_k$ denote the $H^k$--absorbing ball of \eqref{eq:nse:proj:hyper}. Suppose that $I_{m,k}$ interpolates uniformly at scale $h$. Then given $u_0\in\B_k$, for all $1\leq\ell\leq m$, there exists a universal constant $c>0$ such that
    \begin{align}\label{est:Jmk:Hl}
        \sup_{t\geq0}\left(\frac{\Sob{J_{m,k}u(t)}{\dot{H}^\ell}}{\nu}\right)\leq c\left[(\s_{\ell-1}^{1/\ell}+G)^{\ell-1}+\sum_{j=1}^{k}h^{j-\ell}(\s_{j-1}^{1/j}+G)^{j-1}\right]G,
    \end{align} 
where $u$ denotes the corresponding unique strong solution of \eqref{eq:nse:proj:hyper}.
\end{Lem}

\begin{proof}
Since $\ell\geq1$, by the boundedness of the Leray-projection and the triangle inequality, observe that
    \begin{align}\notag
        \Sob{J_{m,k}u}{\dot{H}^\ell}\leq \Sob{I_{m,k}u}{\dot{H}^\ell}\leq \Sob{u-I_{m,k}u}{\dot{H}^\ell}+\Sob{u}{\dot{H}^\ell}.
    \end{align}
From  \cref{cor:IO:global}, \cref{prop:wp:nse} (for $\ell=1$), and \cref{thm:abs:ball:Hk} (for $\ell\geq2$), it follows that
    \begin{align}
        \Sob{u-I_{m,k}u}{\dot{H}^\ell}&\leq  c\sum_{j=1}^{k}h^{j-\ell}\Sob{u}{\dot{H}^{j}}\leq c\sum_{j=1}^{k}h^{j-\ell}(\s_{j-1}^{1/j}+G)^{j-1}G,\notag\\
        \Sob{u}{\dot{H}^{\ell}}&\leq c\nu(\s_{\ell-1}^{1/\ell}+G)^{\ell-1}G.\notag
    \end{align}
Combining these estimates, then taking the supremum over $t\geq0$ yields \eqref{est:Jmk:Hl}.
\end{proof}

\begin{Lem}\label{lem:H1:gal:apriori:app}
Let $\mu>0$ and $k\geq m+1$ such that $k\geq2$. Let
    \begin{align}
        F_1&:=\left[\left(\frac{\nu}{\mu}\right)^2\s_1^2G^2+\til{U}_1^2\right]^{1/2},\label{def:F1}\\
        \til{U}_{1}&:=\frac{\sup_{t\geq0}\Sob{I_{m,k}u(t)}{\dot{H}^1}}{\nu}.\label{def:U1}
    \end{align}
There exists a universal constant $c>0$ such that 
 \begin{align}\label{est:H1:gal:Quniform}
        \frac{d}{dt}\Sob{\nabla v_N}{L^2}^2+&\frac{9}{5{\pi_0}}\sum_q\left(\nu\Sob{\De v_N}{L^2(Q_q)}^2++2\gam\Sob{(-\De)^{p/2+1}v_N}{L^2(Q_q)}^2+\mu\Sob{\nabla v_N}{L^2(Q_q)}^2\right)\notag\\
          &\leq c\mu\nu^2 F_1^2+ c\frac{\mu^2}{\nu}\sum_{j=1}^{k}\sum_q\veps_{0,j}(Q_q)^2h_q^{2j}\Sob{v_N}{\dot{H}^{j}(\til{Q}_q)}^2,
    \end{align}
where $\pi_0$ is the constant from \ref{item:P2}. Moreover, if $I_{m,k}$ interpolates uniformly at scale $h$, then
    \begin{align}\label{est:H1:gal:uniform}
    \frac{d}{dt}\Sob{\nabla v_N}{L^2}^2+\frac{9}{5}\nu\Sob{\De v_N}{L^2}^2+2\gam\Sob{(-\De)^{p/2+1}v_N}{L^2}^2+\frac{9}{5}\mu\Sob{\nabla v_N}{L^2}^2\leq c\mu\nu^2 F_1^2+ c\frac{\mu^2}{\nu}\sum_{j=1}^{k}h^{2j}\Sob{v_N}{\dot{H}^{j}}^2,
    \end{align}
holds for all $\mu>0$, $0<h\leq 2\pi$, and $N>0$.
\end{Lem}

\begin{proof}
Upon taking the $L^2$--inner product of \eqref{eq:nse:ng:gal} with $-\De v_N$, using the identity $\lb P_\s(v_N\cdotp\nabla)v_N,\De v_N\rb=0$, and integrating by parts, we obtain
    \begin{align}
        \frac{1}2\frac{d}{dt}\Sob{\nabla v_N}{L^2}^2+\nu\Sob{\De v_N}{L^2}^2+\gam\Sob{(-\De)^{p/2+1}v_N}{L^2}^2+\mu\Sob{\nabla v_N}{L^2}^2&=\lb \nabla f_\mu, \nabla v_N\rb+\mu\lb v_N-J_{m,k}v_N,\De v_N\rb\notag\\
        &=I^1+II^1.\notag
    \end{align} 
We estimate $I^1$ by applying the Cauchy-Schwarz inequality, Young's inequality, so that
    \begin{align}
        |I^1|&\leq c\nu^3\left[\left(\frac{\nu}{\mu}\right)\left(\frac{\Sob{P_\s f}{\dot{H}^1}}{\nu^2}\right)^2+\left(\frac{\mu}\nu\right)\left(\frac{\Sob{I_{m,k}u}{\dot{H}^1}}{\nu}\right)^2\right]+\frac{\mu}{100}\Sob{\nabla v_N}{L^2}^2\notag\\
        &\leq c\mu\nu^2 \left[\left(\frac{\nu}{\mu}\right)^2\s_1^2G^2+\til{U}_1^2\right]+\frac{\mu}{100}\Sob{\nabla v_N}{L^2}^2.\notag
    \end{align}
    
On the other hand, we treat $II^1$ by applying the Cauchy-Schwarz inequality, \cref{cor:IO:global}, and Young's inequality, to obtain
    \begin{align}\label{est:II1:Quniform}
          |II^1|
        \leq c\frac{\mu^2}{\nu}\sum_{j=1}^{k}\sum_q\veps_{0,j}(Q_q)^2h_q^{2j}\Sob{v_N}{\dot{H}^{j}(\til{Q}_q)}^2+\frac{\nu}{100}\Sob{\De v_N}{L^2}^2.
    \end{align} 
Upon combining $I^1$, \eqref{est:II1:Quniform}, and \cref{lem:multiplicity}, we obtain
    \begin{align}
          \frac{d}{dt}\Sob{\nabla v_N}{L^2}^2+&\frac{9}{5{\pi_0}}\sum_q\left(\nu\Sob{\De v_N}{L^2(Q_q)}^2++2\gam\Sob{(-\De)^{p/2+1}v_N}{L^2(Q_q)}^2+\mu\Sob{\nabla v_N}{L^2(Q_q)}^2\right)\notag\\
          &\leq c\mu\nu^2 \left[\left(\frac{\nu}{\mu}\right)^2\s_1^2G^2+\til{U}_1^2\right]+ c\frac{\mu^2}{\nu}\sum_{j=1}^{k}\sum_q\veps_{0,j}(Q_q)^2h_q^{2j}\Sob{v_N}{\dot{H}^{j}(\til{Q}_q)}^2.\notag 
    \end{align}

On the other hand, if $I_{m,k}$ interpolates uniformly at scale $h$, then
    \begin{align}\label{est:II1:uniform}
        |II^1|&\leq \mu\Sob{v_N-J_{m,k}v_N}{L^2}\Sob{ \De v_N}{L^2}\leq c\frac{\mu^2}{\nu}\Sob{v_N-I_{m,k}v_N}{L^2}\Sob{\De v_N}{L^2}\notag\\
        &\leq c\frac{\mu^2}{\nu}\sum_{j=1}^{k}h^{2j}\Sob{v_N}{\dot{H}^{j}}^2+\frac{\nu}{100}\Sob{\De v_N}{L^2}^2.
    \end{align}
Hence, upon combining $I^1$ and \eqref{est:II1:uniform}, we obtain
    \begin{align}
        \frac{d}{dt}\Sob{\nabla v_N}{L^2}^2+\frac{9}{5}\nu\Sob{\De v_N}{L^2}^2+&2\gam\Sob{(-\De)^{p/2+1}v_N}{L^2}^2+\frac{9}{5}\mu\Sob{\nabla v_N}{L^2}^2\notag\\
        &\leq c\mu\nu^2 \left[\left(\frac{\nu}{\mu}\right)^2\s_1^2G^2+\til{U}_1^2\right]+ c\frac{\mu^2}{\nu}\sum_{j=1}^{k}h^{2j}\Sob{v_N}{\dot{H}^{j}}^2.\notag
    \end{align}
This completes the proof.
\end{proof}

\begin{Lem}\label{lem:Hk:gal:apriori:app}
For all $\ell\geq2$, let
     \begin{align}
        F_{\ell}&:=\left[\left(\frac{\nu}{\mu}\right)^2\s_{\ell-1}^2G^2+\til{U}_{\ell-1}^2\right]^{1/2}\label{def:Fl}\\
        \til{U}_\ell&:=\frac{\sup_{t\geq0}\Sob{I_{m,k}u(t)}{\dot{H}^\ell}}{\nu}\label{def:Ul}.
    \end{align}
Let $m\geq0$, $k\geq2$ be given such that $k\geq m+1$. There exists a universal constant $c>0$ such that 
  \begin{align}\label{est:Hk:gal:app}
     \left(\frac{\Sob{v_N(t)}{\dot{H}^k}}{\nu}\right)^2\leq& \exp\left[c\nu\left(\frac{\mu h^{1-k}}{\nu}\right)^2 t\right]\left[\left(\frac{\Sob{v_0}{\dot{H}^k}}{\nu}\right)^2+h^{2(k-1)} F_{k-1}^2\right]\notag\\
        &+c\nu\int_0^t\exp\left[c\nu\left(\frac{\mu h^{1-k}}{\nu}\right)^2(t-s)\right]\left(\frac{\Sob{\nabla v_N(s)}{L^2}}{\nu}\right)^{2(k+1)}ds,
    \end{align}
holds for all $\mu>0$, $0<h\leq 2\pi$, and $N>0$. 
\end{Lem}

\begin{proof}
    
Now, we estimate in $\dot{H}^k$. By taking the $L^2$--inner product of \eqref{eq:nse:ng:gal} with $(-1)^{|\al|}\bdy^{2\al} v_N$, where $|\al|=k$, integrating by parts, then summing over all $|\al|=k$, we obtain
    \begin{align}\label{eq:nse:ng:gal:Hl}
&\frac{1}2\frac{d}{dt}\Sob{v_N}{\dot{H}^k}^2+\nu\Sob{\nabla v_N}{\dot{H}^k}^2+\gam\Sob{(-\De)^{p/2}\nabla v_N}{\dot{H}^k}^2\notag\\
&=\sum_{|\al|=k}\lb \bdy^\al P_\s f_\mu,\bdy^\al v_N\rb-\mu\sum_{|\al|=k}\lb \bdy^{\al}J_{m,k}v_N,\bdy^{\al} v_N\rb-\sum_{|\al|=k}\lb \bdy^\al((v_N\cdotp\nabla)v_N),\bdy^{\al}v_N\rb\notag\\
&=I^k+II^k+III^k.
    \end{align}

We treat $I^\ell$ by integrating by parts, then applying the Cauchy-Schwarz inequality, Young's inequality, \eqref{def:grashof}, and \eqref{def:shape}, to obtain
    \begin{align}
        |I^k|&\leq c\nu^3\left[\left(\frac{\Sob{P_\s f}{\dot{H}^{k-1}}}{\nu^2}\right)^2+\left(\frac{\mu}{\nu}\right)^2\left(\frac{\Sob{I_{m,k}u}{\dot{H}^{k-1}}}{\nu}\right)^2\right]+\frac{\nu}{100}\Sob{\nabla v_N}{\dot{H}^k}^2\notag\\
        &\leq c\mu^2\nu\left[\left(\frac{\nu}{\mu}\right)^2\s_{k-1}^2G^2+\til{U}_{k-1}^2\right]+\frac{\nu}{100}\Sob{\nabla v_N}{\dot{H}^k}^2.\notag
    \end{align}

Next, we treat $II^k$. We observe that $\bdy^\al J_{m,k}\phi=\bdy^\al I_{m,k}\phi$, then apply the Cauchy-Schwarz inequality, \eqref{eq:IO:loc:bdd}, \cref{cor:IO:global}, Poincar\'e's inequality, and Young's inequality to obtain
    \begin{align}\label{est:IIl:a}
        |II^k|&\leq c\mu\Sob{I_{m,k}v_N}{\dot{H}^k}\Sob{v_N}{\dot{H}^k}\leq c\mu h^{m-k}\Sob{I_{m,k}v_N}{\dot{H}^{m}}\Sob{\nabla v_N}{\dot{H}^k}\notag\\
        &\leq c\mu h^{m-k}\left(\Sob{v_N}{\dot{H}^{m}}+\Sob{v_N-I_{m,k}v_N}{\dot{H}^{m}}\right)\Sob{\nabla v_N}{\dot{H}^k}\notag\\
        &\leq c\mu h^{m-k}\left(\Sob{v_N}{\dot{H}^{k}}+\sum_{j=1}^{k}h^{j-m}\Sob{v_N}{\dot{H}^{j}}\right)\Sob{\nabla v_N}{\dot{H}^k}\notag\\
        &\leq c\frac{\mu^2}{\nu} h^{2(1-k)}\Sob{v_N}{\dot{H}^k}^2+\frac{\nu}{100}\Sob{\nabla v_N}{\dot{H}^k}^2.
    \end{align}

Lastly, we estimate $III^k$. Indeed, observe that due to the divergence free condition, we have
    \begin{align}\notag
        \lb\bdy^\al((v_N\cdotp\nabla)v_N),\bdy^\al v_N\rb=\lb [\bdy^\al, v_N\cdotp\nabla]v_N,\bdy^\al v_N\rb.
    \end{align}
Thus, by H\"older's inequality, a classical commutator estimate (see for instance \cite{KatoPonce1988, KenigPonceVega1991}), and interpolation we have
    \begin{align}
        |III^k|&\leq\sum_{|\al|=k}|\lb[\bdy^\al,v_N\cdotp\nabla]v_N,\bdy^\al v_N\rb|\leq c\sum_{|\al|=\ell}\Sob{[\bdy^\al, v_N\cdotp\nabla]v_N}{L^{4/3}}\Sob{\bdy^\al v_N}{L^4}\notag\\
        &\leq c\sum_{|\al|=k}\Sob{\bdy^\al v_N}{L^4}^2\Sob{\nabla v_N}{L^2}\leq c\Sob{\nabla v_N}{\dot{H}^k}\Sob{v_N}{\dot{H}^k}\Sob{\nabla v_N}{L^2}\notag\\
        &\leq c\Sob{\nabla v_N}{\dot{H}^k}^{\frac{2k-1}k}\Sob{\nabla v_N}{L^2}^{\frac{k+1}{k}}\leq \frac{\nu}{100}\Sob{\nabla v_N}{\dot{H}^k}^2+c\nu^3 \left(\frac{\Sob{\nabla v_N}{L^2}}{\nu}\right)^{2(k+1)}.\notag
    \end{align}

Upon returning to \eqref{eq:nse:ng:gal:Hl} and combining $I^k$, $II^k$, and $III^k$, we have that
  \begin{align}\label{est:Hk:gal:pregronwall}
       \frac{d}{dt}\Sob{v_N}{\dot{H}^k}^2+\frac{9}5\nu\Sob{\nabla v_N}{\dot{H}^k}^2+&2\gam\Sob{(-\De)^{p/2}\nabla v_N}{\dot{H}^k}^2\notag\\
       &\leq  c\mu^2\nu F_{k-1}^2+c\nu^3 \left(\frac{\Sob{\nabla v_N}{L^2}}{\nu}\right)^{2(k+1)}+c\nu\left(\frac{\mu h^{1-k}}{\nu}\right)^2 \Sob{v_N}{\dot{H}^k}^2.
    \end{align}
Hence, by Gronwall's inequality, it follows that
    \begin{align}\label{est:Hk:gal:postgronwall}
        \Sob{v_N(t)}{\dot{H}^k}^2\leq& \exp\left[c\nu\left(\frac{\mu h^{1-k}}{\nu}\right)^2 t\right]\left(\Sob{v_0}{\dot{H}^k}^2+\nu^2h^{2(k-1)}F_{k-1}^2\right)\notag\\
        &+c\nu^3\int_0^t\exp\left[c\nu\left(\frac{\mu h^{1-k}}{\nu}\right)^2(t-s)\right]\left(\frac{\Sob{\nabla v_N(s)}{L^2}}{\nu}\right)^{2(k+1)}ds,
    \end{align}
as desired.
\end{proof}

Under certain assumptions on $I_{m,k}$, one may then identify conditions on $\mu$ so that the sequence $\{v_N\}_{N>0}$ is bounded uniformly in time in $\dot{H}^1$, independent of $N$. Provided that the initial data belongs to an absorbing ball for the dynamics, these bounds can then be expressed explicitly in terms of its radius.

\begin{Cor}\label{cor:H1:gal:apriori:app}
Suppose that $1+m\leq k\leq 2+p$. Then there exists a universal constant $c>0$, independent of $N>0$, such that if $\mu>0$ and $\{h_q\}_q$ satisfy
    \begin{align}\label{cond:H1:mu:h:partition:mult:app}
     c\sup_q\frac{\mu h_q^2}{\nu}\left[\veps_{0,1}(Q_q)+\veps_{0,2}(Q_q)+\chi_{(0,\infty)}(\gam)\left(\frac{\nu}{\gam}\right)\left(\frac{\mu h_q^2}{\nu}\right)\sum_{j=1}^{[p]}\veps_{0,j}(Q_q)^2h_q^{2(j-2)}\right]\leq\frac{1}{10{\pi_0}},
    \end{align}
then
    \begin{align}\label{est:H1:gal:partition:mult:app}
    \left(\frac{\Sob{\nabla v_N(t)}{L^2}}{\nu}\right)^2\leq e^{-\frac{\mu}{\pi_0}t}\left(\frac{\Sob{\nabla v_0}{L^2}}{\nu}\right)^2+c{\pi_0}F_1^2,\quad \text{for all}\ t\geq0.
    \end{align}
    
On the other hand, if $I_{m,k}$ interpolates uniformly at scale $h$, then there exists a universal constant $c>0$, independent of $N>0$, such that if $\mu>0$ and $0<h\leq2\pi$ satisfy
    \begin{align}\label{cond:H1:mu:h:uniform:app}
        c\frac{\mu h^2}{\nu}\left[1+\chi_{(0,\infty)}(\gam)\left(\frac{\mu }{\gam}\right)\sum_{j=1}^{[p]}h^{2(j-1)}\right]\leq\frac{1}{10}, 
    \end{align}
where $[p]$ denotes the greatest integer $\leq p$, then
    \begin{align}\label{est:H1:gal:uniform:app}
     \left(\frac{\Sob{\nabla v_N(t)}{L^2}}{\nu}\right)^2\leq e^{-\mu t}\left(\frac{\Sob{\nabla v_0}{L^2}}{\nu}\right)^2+cF_1^2,\quad \text{for all}\ t\geq0.
    \end{align}
In particular, if $v_0\in\B_1$ and $u_0\in\B_1\cap\B_k$, then 
    \begin{align}\label{est:H1:gal:final:app}
        \sup_{t\geq0}\left(\frac{\Sob{\nabla v_N(t)}{L^2}}{\nu}\right)^2\leq c\left[1+\left(\frac{\nu}{\mu}\right)^2\s_1+\sum_{j=1}^{k}(\s_{j-1}^{1/j}+G)^{2(j-1)}\right]G^2,
    \end{align}
for all $N>0$.

\end{Cor}

\begin{proof}
By \eqref{est:H1:gal:partition:mult:app} and the Poincar\'e inequality we have
    \begin{align}
    &\frac{d}{dt}\Sob{\nabla v_N}{L^2}^2+\gam\left(\frac{18}{5{\pi_0}}-c\frac{\mu^2}{\gam\nu}\sum_{j=3}^{2+[p]}\veps_{0,j}(Q_q)^2h_q^{2j}\right)\Sob{(-\De)^{p/2+1}v_N}{L^2(Q_q)}^2\notag\\
      &\leq-\nu\sum_q\left[\frac{9}{5{\pi_0}}-c\left(\frac{\mu h_q^2}{\nu}\right)^2\veps_{0,2}(Q_q)^2\right]\Sob{\De v_N}{L^2(Q_q)}^2-\mu\sum_q\left(\frac{9}{5{\pi_0}}-c\frac{\mu h_q^2}{\nu}\veps_{0,1}(Q_q)^2\right)\Sob{\nabla v_N}{L^2(Q_q)}^2+ c\mu\nu^2 F_1^2.\notag
    \end{align}
Since \eqref{cond:H1:mu:h:partition:mult:app} holds and $k\leq 2+p$, it then follows upon shifting the index and upon applying \cref{lem:multiplicity} that
    \begin{align}
        \frac{d}{dt}\Sob{\nabla v_N}{L^2}^2\leq -\frac{3}{2{\pi_0}}\mu\Sob{\nabla v_N}{L^2}^2+c\mu\nu^2 F_1^2.\notag
    \end{align}
We deduce from Gronwall's inequality that
    \begin{align}
        \Sob{\nabla v_N(t)}{L^2}^2\leq e^{-\frac{\mu}{\pi_0}t}\Sob{\nabla v_0}{L^2}^2+cM\nu^2F_1^2,\notag
    \end{align}
as desired.

On the other hand, if $I_{m,k}$ interpolates uniformly at scale $h$, we apply \cref{lem:H1:gal:apriori:app}, so that by \eqref{est:H1:gal:uniform} we have
    \begin{align}\notag
         \frac{d}{dt}\Sob{\nabla v_N}{L^2}^2+\nu\left[\frac{9}{5}-c\left(\frac{\mu h^2}{\nu}\right)^2\right]\Sob{\De v_N}{L^2}^2+&\gam\left(2-c\frac{\mu^2}{\gam\nu}\sum_{j=3}^{2+[p]}h^{2j}\right)\Sob{(-\De)^{p/2+1}v_N}{L^2}^2\notag\\
         &\leq -\mu\left(\frac{9}{5}-c\frac{\mu h^2}{\nu}\right)\Sob{\nabla v_N}{L^2}^2+c\mu\nu^2 F_1^2,\notag
    \end{align}
Since \eqref{cond:H1:mu:h:uniform:app} holds, it follows that
    \begin{align}\notag
        \frac{d}{dt}\Sob{\nabla v_N}{L^2}^2\leq-\frac{3}2\mu\Sob{\nabla v_N}{L^2}^2+ c\mu\nu^2 F_1.
    \end{align}
An application of Gronwall's inequality, then yields
    \begin{align}\notag
        \Sob{\nabla v_N(t)}{L^2}^2\leq e^{-\mu t}\Sob{\nabla v_0}{L^2}^2+c\nu^2F_1^2,
    \end{align}
which implies \eqref{est:H1:gal:uniform:app}. We deduce \eqref{est:H1:gal:final:app} by applying \cref{lem:Jmk:Hl} and \cref{prop:wp:nse}. 
\end{proof}

Upon combining \cref{lem:Hk:gal:apriori:app} and \cref{cor:H1:gal:apriori:app}, we obtain the following corollary.

\begin{Cor}\label{cor:Hk:gal:apriori:app}
Suppose that $1+m\leq k\leq 2+p$. If $\mu$ and $\{h_q\}_q$ satisfy \eqref{cond:H1:mu:h:partition:mult:app}, then
    \begin{align}\label{est:Hk:gal:energy:app}
         \sup_{N>0}\sup_{0\leq t\leq T}\left[\Sob{v_N(t)}{\dot{H}^k}^2+\nu\int_0^t\Sob{\nabla v_N(s)}{\dot{H}^k}^2+\gam\int_0^t\Sob{(-\De)^{p/2}\nabla v_N(s)}{\dot{H}^k}^2ds\right]<\infty,
    \end{align}
holds for all $T>0$. Moreover, if $I_{m,k}$ interpolates uniformly at scale $h$ and $\mu, h$ satisfy \eqref{cond:H1:mu:h:uniform:app} in place of \eqref{cond:H1:mu:h:partition:mult:app}, then \eqref{est:Hk:gal:energy:app} still holds.
\end{Cor}

\begin{proof}
We directly apply  \eqref{est:H1:gal:partition:mult:app} or \eqref{est:H1:gal:uniform:app} in \eqref{est:Hk:gal:app}, to deduce uniform-in-time bounds on $\Sob{v_N}{\dot{H}^k}$. Upon returning to \eqref{est:Hk:gal:pregronwall}, we integrate over $[0,T]$, then apply the uniform-in-time bounds just obtained for $\Sob{v_N}{\dot{H}^k}$ to deduce \eqref{est:Hk:gal:energy:app}.
\end{proof}

Lastly, we establish bounds for the time derivative $\frac{dv_N}{dt}$.

\begin{Lem}\label{lem:gal:ddt:app}
Under the assumptions of \cref{cor:Hk:gal:apriori:app}, we have
    \begin{align}
        \sup_{N>0}\int_0^T\Sob{\frac{dv_N}{dt}(s)}{L^2}^2ds<\infty,\notag
    \end{align}
for all $T>0$.
\end{Lem}

\begin{proof}
Observe that
    \begin{align}
        \Sob{\frac{dv_N}{dt}}{L^2}&\leq \Sob{v_N(t)}{\dot{H}^2}+\Sob{v_N\cdotp\nabla v_N}{L^2}+\Sob{f_\mu}{L^2}+\mu\Sob{J_{m,k}v_N}{L^2}\notag\\
        &\leq I+II+III+IV\notag
    \end{align}
We treat $I$ by Poincar\'e's inequality and \cref{cor:Hk:gal:apriori:app}. For $II$, we apply Cauchy-Schwarz inequality and interpolation to obtain
    \begin{align}
        |II|\leq \Sob{v_N}{L^4}\Sob{\nabla v_N}{L^4}\leq\Sob{v_N}{L^2}\Sob{v_N}{\dot{H}^2}.\notag
    \end{align} 
Thus, we may ultimately control $II$ with Poincar\'e's inequality and \cref{cor:Hk:gal:apriori:app}. We treat $III$ with the Cauchy-Schwarz inequality and \cref{lem:Jmk:Hl}. Lastly, we treat $IV$ with the Poincar\'e inequality, \cref{cor:IO:global:bdd}, and the fact that $1+m\leq k\leq 2+p$, so that we have
    \begin{align}
        |IV|\leq \mu\Sob{I_{m,k}v_N}{\dot{H}^1}\leq c\mu\Sob{v_N}{\dot{H}^{3+p}}\leq c\mu\Sob{(-\De)^{p/2}\nabla v_N}{\dot{H}^k}.\notag
    \end{align}
Hence, by \cref{cor:Hk:gal:apriori:app}, we conclude that
    \begin{align}
        \sup_{N>0}\int_0^T\Sob{\frac{dv_N}{dt}(s)}{L^2}^2ds<\infty,\notag
    \end{align}
for all $T>0$.
\end{proof}

Finally, we are ready to prove \cref{thm:wp:ng}.

\begin{proof}[Proof of \cref{thm:wp:ng}]
By \cref{cor:Hk:gal:apriori:app} and \cref{lem:gal:ddt:app}, we may apply the Aubin-Lions compactness theorem to extract a subsequence of $\{v_N\}_{N>0}$ such that $v_{N'}\goesto v$ in $L^2(0,T;\dot{H}^k_\s)$, where $v\in L^\infty(0,T;\dot{H}^k)\cap L^2(0,T;\dot{H}^{k+1})\cap L^2(0,T;\dot{H}^{k+1+p})$. This is sufficient to pass to the limit in \eqref{eq:nse:ng:proj:hyper} and show that $v$ indeed satisfies the equation. Moreover, by \cref{lem:gal:ddt:app}, we also have $dv/dt\in L^2(0,T;\dot{H}^{-k})$. Hence $v\in C([0,T];\dot{H}^{k})$, so that, in conjunction with $v\in L^\infty(0,T;\dot{H}^k)$, we deduce that $v\in C([0,T];\dot{H}^k)$. Uniqueness of solutions follows in the same way as in \cite{AzouaniOlsonTiti2014}; the relevant details can be inferred from the analysis performed above in establishing the synchronization of solutions (see Proof of \cref{thm:sync:Hk}).
\end{proof}

\section{Taylor interpolant}\label{sect:app:taylor}

We prove a preliminary lemma, which is a generalization of that found in \cite{JonesTiti1992a}.

\begin{Lem}\label{lem:trace}
Let $h>0$ and $d\geq2$.  Let $Q=[0,h]^d$ and $\phi\in C^k(Q)$, where $0\leq k\leq d$. For each $1\leq k\leq d-1$, there exist universal constants $b_\al>0$, for each multi-index $0\leq|\al|\leq k$, depending only on $d$, such that 
		\begin{align}\label{eq:trace:gen}
			\sup_{y\in[0,h]^{d-k}}\Sob{\phi(\cdotp, y)}{L^2([0,h]^{k})}^2\leq \sum_{0\leq|\al|\leq k} b_{|\al|} h^{-d+k+2|\al|}\Sob{\bdy^\al\phi}{L^2(Q)}^2.
		\end{align}
\end{Lem}

\begin{proof}
Let $k=d-1$. For $y_d\leq x_d$, we have
	\begin{align}\notag
		\phi(x_1,\dots, x_{d-1}, x_d)-\phi(x_1,\dots,x_{d-1}, y_d)=\int_{y_d}^{x_d}(\bdy_{d}\phi)(x_1,\dots, x_{d-1}, \tau)\ d\tau.
	\end{align}
By applying the Cauchy-Schwarz inequality, then integrating with respect to $dx_1\dots dx_d$ over $\Om$, we deduce
	\begin{align}\notag
		h\Sob{\phi(\cdotp, y_d)}{L^2([0,h]^{d-1})}^2\leq 2\Sob{\phi}{L^2(\Om)}^2+2h^2\Sob{\nabla\phi}{L^2(\Om)}^2,
	\end{align}
for all $y_d\in[0,h]$. Dividing by $h^{d-1}$ establishes \eqref{eq:trace:gen} for $k=1$. Observe that it now suffices to assume $d\geq3$.

We argue by induction. In particular, suppose that \eqref{eq:trace:gen} holds for all $\ell=k,\dots, d-1$, for some $1<k\leq d-1$.  We show that \eqref{eq:trace:gen} holds for $k-1$. Let us denote by $\hat{z}_\ell$ the projected point $(z_{\ell},\dots,z_d)$, where $1\leq \ell\leq d$.  Observe that $z=\hat{z}_1=(z_1,\dots, z_{l}, \hat{z}_{l+1})$, for all $l=1,\dots, d-1$, and  $\hat{z}_d=z_d$. Fix $x\in\Om$. Given any $y\in \Om$, we denote $\phi(\cdotp,\hat{y}_{\ell})=\phi(x_1,\dots, x_{\ell-1},\hat{y}_{\ell})$, for $\ell=1,\dots, d-1$. We have
	\begin{align}
		\phi(\cdotp,\hat{x}_{k})-\phi(\cdotp,\hat{y}_{k})=&(\phi(\cdotp, x_{d-1},x_d)-\phi(\cdotp,x_{d-1}, \hat{y}_d))+(\phi(\cdotp, x_{d-1}, \hat{y}_d)-\phi(\cdotp, y_{d-1},\hat{y}_{d}))\notag\\
			&+\dots+(\phi(\cdotp, x_{k}, \hat{y}_{k+1})-\phi(\cdotp,y_{k}, \hat{y}_{k+1})).\notag
	\end{align}
It follows that
	\begin{align}\label{mean:value}
		\phi(\cdotp,\hat{x}_{k})-\phi(\cdotp,\hat{y}_{k})=\int_{y_d}^{x_d}{(\bdy_{d}\phi)(\cdotp, \tau)}\ d\tau&+\int_{y_{d-1}}^{x_{d-1}}(\bdy_{d-1}\phi)(\cdotp, \tau, \hat{y}_d)\ d\tau\notag\\
				+\dots&+\int_{y_{k}}^{x_{k}}(\bdy_{k}\phi)(\cdotp,\tau, \hat{y}_{k+1})\ d\tau.
	\end{align}
Applying Cauchy-Schwarz yields
	\begin{align}\notag
		\abs{\phi(\cdotp, \hat{y}_{k})}^2&\leq (d-k+2)\left[\abs{\phi(x)}^2+h\int_{0}^{h}\abs{\bdy_{d}\phi(\cdotp,\tau)}^2\ d\tau+h\sum_{j=0}^{d-k-1}\int_{0}^{h}\abs{\bdy_{k+j}\phi(\cdotp,\tau,\hat{y}_{k+j+1})}^2\ d\tau\right]\notag\\
		&=(d-k+2)\left[\abs{\phi(x)}^2+h\int_{0}^{h}\abs{\bdy_{d}\phi(\cdotp,\tau)}^2\ d\tau+h\sum_{j=k}^{d-1}\int_{0}^{h}\abs{\bdy_{j}\phi(\cdotp,\tau,\hat{y}_{j+1})}^2\ d\tau\right].\notag
	\end{align}
Integrating with respect to $dx_1\dots dx_{d}$ over $\Om$ then gives  
	\begin{align}\label{prep:induct:step}
		h^{d-k+1}\Sob{\phi(\cdotp, \hat{y}_{k})}{L^2([0,\ell]^{k-1})}^2
		&\leq c\left(\Sob{\phi}{L^2(\Om)}^2+h^2\Sob{\nabla\phi}{L^2(\Om)}^2+\sum_{j=k}^{d-1}h^{d-j+3}\Sob{\nabla\phi(\cdotp,\hat{y}_{j})}{L^2([0,h]^{j-1})}^2\right).
	\end{align}
It follows from the induction hypothesis that
	\begin{align}\label{induct:step}
		\Sob{\nabla\phi(\cdotp,\hat{y}_{j})}{L^2([0,h]^{j-1})}^2\leq \sum_{0\leq|\al|\leq j-1}c_{|\al|}h^{-d+j-1+2|\al|}\Sob{\bdy^\al\nabla\phi}{L^2(\Om)}^2.
	\end{align}
Therefore, upon substituting the bounds in \eqref{induct:step} into \eqref{prep:induct:step}, then combining like terms we arrive at
	\begin{align}\notag
		h^{d-k+1}\Sob{\phi(\cdotp, \hat{y}_k)}{L^2([0,\ell]^{d-k})}^2\leq c_0\Sob{\phi}{L^2(\Om)}^2+\sum_{|\al|=0}^{d-1} b_{|\al|}h^{2+2|\al|}\Sob{\bdy^\al\nabla\phi}{L^2({\Om})}^2.
	\end{align}
The proof is complete upon dividing by $h^{d-k+1}$.
\end{proof}

\begin{Prop}
Suppose $d\geq2$. Let $Q=[0,h]^d$ and $\phi\in C^1(Q)$.  Given $y\in Q$, let $T_1\phi(\cdotp;y)$ denote first-order Taylor polynomial of $\phi$ centered at $y$. There exists an absolute constant $C>0$, independent of $y$, such that
    \begin{align}\label{est:taylor:1}
     \Sob{\phi-T_1\phi(\cdotp;y)}{L^2(Q)}^2\leq \sum_{1\leq|\al|\leq d}c_{|\al|}h^{2(|\al|+1)}\Sob{\bdy^\al\nabla\phi}{L^2(Q)}^2.
    \end{align}
\end{Prop}

\begin{proof}
    Let $(x_1,x_2, (y_1,y_2)\in Q$.  Then observe that 
	\begin{align}
		\phi&(x_1,x_2)-T_{1}\phi(x_1,x_2)=\phi(x_1,x_2)-\phi(y_1,y_2)-\nabla\phi(y_1,y_2)\cdotp(x-y_1,y-y_2)\notag\\
			=&\left(\phi(x_1,x_2)-\phi(x,y_2)\right)+\left(\phi(x,y_2)-\phi(y_1,y_2)\right)-\nabla\phi(y_1,y_2)\cdotp(x-y_1,y-y_2)\notag\\
			=&\int_{y_2}^y(\bdy_y\phi)(x,s)\ ds+\int_{y_1}^x(\bdy_x\phi)(s,y_2)\ ds-(\bdy_x\phi)(y_1,y_2)(x-y_1)-(\bdy_y\phi)(y_1,y_2)(y-y_2)\notag\\
			=&\int_{y_2}^y(\bdy_y\phi)(x,s)-(\bdy_y\phi)(x,y_2)\ ds+\left[(\bdy_y\phi)(x,y_2)-(\bdy_y\phi)(y_1,y_2)\right](y-y_2)\notag\\
				&+\int_{y_1}^x(\bdy_x\phi)(s,y_2)-(\bdy_x\phi)(y_1,y_2)\ ds\notag\\
			=&\int_{y_2}^y\int_{y_2}^s(\bdy_2^2\phi)(x,\tau)\ d\tau ds+\int_{y_2}^y\int_{y_1}^x(\bdy_2\bdy_1\phi)(\tau,y_2)\ d\tau ds+\int_{y_1}^x\int_{y_1}^s(\bdy_1^2\phi)(\tau, y_2)\ d\tau ds.\notag
	\end{align}
	It follows from H\"older's inequality that
	\begin{align}
		&|\phi(x_1,x_2)-T_1\phi(x_1,x_2)|\notag\\
	&\leq\int_{y_2}^y\int_{y_2}^y|\bdy_2^2\phi(x,\tau)|\ d\tau ds+\int_{y_2}^y\int_{y_1}^x|\bdy_2\bdy_1\phi(\tau,y_2)|\ d\tau ds+\int_{y_1}^x\int_{y_1}^x|\bdy_1^2\phi(\tau,y_2)|\ d\tau ds\notag\\
		&\leq |y-y_2|h^{1/2}\left(\Sob{\bdy_2^2\phi(x,\cdotp)}{L^2(0,h)}+\Sob{\bdy_2\bdy_1\phi(\cdotp, y_2)}{L^2(0,h)}\right)+|x-y_1|h^{1/2}\Sob{\bdy_1^2\phi(\cdotp, y_2)}{L^2(0,h)}\notag\\
		&\leq h^{3/2}\left(\Sob{\bdy_2^2\phi(x,\cdotp)}{L^2(0,h)}+\Sob{\bdy_2\bdy_1\phi(\cdotp, y_2)}{L^2(0,h)}+\Sob{\bdy_1^2\phi(\cdotp, y_2)}{L^2(0,h)}\right)\notag
	\end{align}
The Cauchy-Schwarz inequaliy then implies that
	\begin{align}\label{t1:eqn1}
		\Sob{\phi-T_1\phi}{L^2(Q)}^2\leq ch^{4}\Sob{\bdy_2^2\phi}{L^2(Q)}^2+ch^{5}\left(\Sob{\bdy_2\bdy_1\phi(\cdotp, y_2)}{L^2(0,h)}^2+\Sob{\bdy_1^2\phi(\cdotp, y_2)}{L^2(0,h)}^2\right).
	\end{align}

By Lemma \ref{lem:trace}, we have
	\begin{align}\notag
		\Sob{\psi(\cdotp, y_2)}{L^2(0,h)}^2\leq b_0h^{-1}\Sob{\psi}{L^2(Q)}^2+b_1h\Sob{\nabla\psi}{L^2({Q})}^2.
	\end{align}
We apply this to $\psi=\bdy_2\bdy_1\phi, \bdy_1^2\phi$, so that \eqref{t1:eqn1} becomes
	\begin{align}\notag
			\Sob{\phi-T_1\phi}{L^2(Q)}^2\leq c_0h^{4}\sum_{j=1,2}\Sob{\bdy_i\nabla\phi}{L^2(Q)}^2+c_1h^6\sum_{i,j=1,2}\Sob{\bdy_i\bdy_j\nabla\phi}{L^2(Q)}^2,
	\end{align}
as desired, which establishes the case $d=2$. 

Now suppose $d\geq3$ and let $x=(x_1,\dots, x_d)$ and $y_2=(y_1,\dots, y_d)$, where $x,x\in Q$. For convenience, in addition to the notation $\hat{z}_\ell$ introduced in the proof of \cref{lem:trace}, we define $\bar{z}_\ell=(z_1,\dots, z_\ell)$, so that $z=(\bar{z}_\ell,\hat{z}_{\ell+1})$. By the fundamental theorem of calculus, we have
    \begin{align}\label{eq:taylor:diff:1}
        \phi&(x)-T_{1}\phi(x;y)=\phi(x)-\phi(y)-\nabla\phi(y)\cdotp(x-y)\notag\\
        =&(\phi(\bar{x}_1,\hat{x}_2)-\phi(\bar{y}_1,\hat{x}_2))+(\phi(\bar{y}_1,x_2,\hat{x}_3)-\phi(\bar{y}_1,y_2,\hat{x}_3))+(\phi(\bar{y}_2,x_3,\hat{x}_4)-\phi(\bar{y}_2,y_3,\hat{x}_4)\notag\\
        &+\dots+(\phi(\bar{y}_{d-2},y_{d-1},x_d)-\phi(\bar{y}_{d-2},y_{d-1},y_d))-\sum_{j=1}^d\bdy_j\phi(y)(x_j-y_j)\notag\\
        =&\int_{y_1}^{x_1}\left(\bdy_1\phi(s_1,\hat{x}_2)-\bdy_1\phi(y_1,\hat{y}_2)\right)ds_1+\int_{y_2}^{x_2}(\bdy_2\phi(y_1,s_2,\hat{x}_3)-\bdy_2\phi(y_1,y_2,\hat{y}_3))ds_2\notag\\
        &+\int_{y_3}^{x_4}(\bdy_3\phi(\bar{y}_2,s_3,\hat{x}_4)-\bdy_3\phi(\bar{y}_2,y_3,\hat{y}_4))ds_3+\dots+\int_{y_d}^{x_d}(\bdy_d\phi(\bar{y}_{d-1},s_d)-\bdy_d\phi(\bar{y}_{d-1},y_d))ds_d.
    \end{align}
Let us interpret $\bar{y}_0$ as an empty position and let $\psi(\cdotp,\hat{z}_\ell)=\psi(\bar{y}_{\ell-1},\hat{z}_\ell)$, for all $\ell=1,\dots,d$. Then for $1\leq j\leq d-2$, we have
    \begin{align}\label{eq:taylor:diff:2a}
       &\bdy_j\phi(\cdotp,s_j,\hat{x}_{j+1})-\bdy_j\phi(\cdotp,y_j,\hat{y}_{j+1})=\bdy_j\phi(\cdotp,s_j,\hat{x}_{j+1})-\bdy_j\phi(\cdotp,y_j,\hat{x}_{j+1})\notag\\
       &+(\bdy_j\phi(\cdotp,x_{j+1},\hat{x}_{j+2})-\bdy_j\phi(\cdotp,y_{j+1},\hat{x}_{j+2}))+\dots+(\bdy_j\phi(\cdotp,x_d)-\bdy_j\phi(\cdotp,y_d))\notag\\
       &=\int_{y_j}^{s_j}\bdy_j^2\phi(\cdotp,s,\hat{x}_{j+1})ds+\int_{y_{j+1}}^{x_{j+1}}\bdy_{j+1}\bdy_j\phi(\cdotp,s,\hat{x}_{j+2})ds+\dots+\int_{y_d}^{x_d}\bdy_d\bdy_j\phi(\cdotp,s)ds.
    \end{align}
Similarly, for $j=d-1,d$, we have
    \begin{align}\label{eq:taylor:diff:2b}
        \bdy_{d-1}\phi(\cdotp,s_{d-1},\hat{x}_d)-\bdy_{d-1}\phi(\cdotp,y_{d-1},\hat{y}_d)&=\int_{y_{d-1}}^{s_{d-1}}\bdy_{d-1}^2\phi(\cdotp,s,\hat{x}_d)ds+\int_{y_d}^{x_d}\bdy_d\bdy_{d-1}\phi(\cdotp,s)ds\notag\\
        \bdy_{d}\phi(\cdotp,s_d)-\bdy_{d}\phi(\cdotp,y_d)&=\int_{y_d}^{s_d}\bdy_d^2\phi(\cdotp,s)ds.
    \end{align}
Now interpret $\hat{z}_{d+1}$ as the empty position. Suppose $z_k\in\{s_k,x_k\}$ and observe that for $1\leq j\leq k\leq d$, upon applying the Cauchy-Schwarz inequality, then integrating over $Q$ with respect to $dx_1\dots dx_d$, we obtain
    \begin{align}\label{est:taylor:cs}
       \int_{Q} \left(\int_{y_j}^{x_j}\int_{y_{k}}^{z_{k}}\bdy_k\bdy_j\phi(\cdotp,s,\hat{x}_{k+1})ds\right)^2dx\leq h^{3+k}\Sob{\bdy_k\bdy_j\phi(\bar{y}_{k-1},\cdotp)}{L^2([0,h]^{d-k+1})}^2.
    \end{align}
Upon returning to \eqref{eq:taylor:diff:1}, applying \eqref{eq:taylor:diff:2a}, \eqref{eq:taylor:diff:2b}, taking the square of the result, integrating over $[0,h]^d$ with respect to $dx_1\dots dx_d$, then applying the Cauchy-Schwarz inequality and \eqref{est:taylor:cs}, we have
    \begin{align}\notag
        \Sob{\phi-T_1\phi}{L^2(Q)}^2\leq c_{(1,1)}h^4\Sob{\bdy_1^2\phi}{L^2(Q)}^2+\sum_{\substack{j\leq k\\(i,j)\neq(1,1)}}c_{(j,k)}h^{3+k}\Sob{\bdy_k\bdy_j\phi(\bar{y}_{k-1},\cdotp)}{L^2([0,h]^{d-k+1})}^2.
    \end{align}
Finally, we apply \cref{lem:trace} to obtain
    \begin{align}\notag
    \Sob{\phi-T_1\phi}{L^2(Q)}^2\leq c_{(1,1)}h^4\Sob{\bdy_1^2\phi}{L^2(Q)}^2+\sum_{\substack{j\leq k\\(i,j)\neq(1,1)}}\sum_{0\leq|\al|\leq d-k+1}c_{(j,k)}b_{|\al|}h^{4+2|\al|}\Sob{\bdy^\al\bdy_k\bdy_j\phi}{L^2(Q)}^2.
    \end{align}
Switching the order of summation completes the proof.
\end{proof}

\section{Volume Elements} \label{sect:app:vol}
We describe an approximation operator based on data given by integration over subsets of each cell. In particular we construct the operator on the unit cube $[0,1]^d$, from which its definition on affine images in the domain follows. For this particular section, we refer the interested reader to \cite{BrownThesis} for additional details. 

First we define an index set and collection of subsets of the cube
\begin{align}\notag
\mathcal{A}_m := \left\{ \al \in \{0,..,m-1\}^d \right\}  \quad \text{and} \quad E_\al := \tfrac{1}{m}\left( \al + [0,1]^d\right) = \prod_{i=1}^d \big[\tfrac{\al_i}{m},\tfrac{\al_i + 1}{m}\big]. 
\end{align}
The degrees of freedom are then given by integration on the subsets. Define the set of functionals $\Sigma = \{\s_\al\}_{\al \in \mathcal{A}_m}$, where $\s_\al: L^1_{\text{loc}}([0,1]^d) \to \R$ are given by
\begin{align}\notag
\s_\al(\phi) =  \int_{E_\al} \phi(x) dx.
\end{align}

Now let us recall that unisolvence of a function space $X$ with respect to a collection of functionals $\Sigma$ is equivalent to $\Sigma$ forming a basis for the dual space of $X$. Unisolvence will ensure that the approximation operator constructed from the functionals will act as identity on $X$. Prior to proving unisolvence of the tensor product volume element in general, we begin in one dimension.
\begin{Lem} \label{lem:vol:uni1d}
Let $m \geq 1$. Then $\mathcal{P}_{m-1,1} = \left\{ \sum_{k=0}^{m-1} p_k x^k : p_k \in \R \right\}$ is unisolvent with respect to the functionals $\Sigma = \{\s_k \}_{k=0}^{m-1}$ given by $\s_k(f) = \int_k^{k+1} f(x) dx$. In particular, there exists a unique set $\Tht = \{\tht_\ell \}_{\ell=0}^{m-1} \subseteq \mathcal{P}_{m-1,1}$ such that $\s_k(\tht_\ell) = \de_\ell^k$.
\end{Lem}
\begin{proof}
For each functional $\s_k$ and monomial $x^j$, where $0\leq k,j \leq m-1$, we have
\begin{align}\notag
\s_k(x^j) = \int_k^{k+1} x^j dx= \frac{(k+1)^{j+1}-k^{j+1}}{j+1}.\notag
\end{align}
Define a matrix $M$ by $M_{kj} = \s_{k-1}(x^{j-1})$
\begin{align}\notag
M &= \left( \frac{k^j-(k-1)^j}{j} \right)_{kj}.
\end{align}
The matrix $\h{M} =  \left( k^j-(k-1)^j \right)_{kj}$ has the same determinant as $M$ up to a factor of $j$ for each row:
\begin{align}\notag
\det(M) &= \left( \prod_{j=1}^m j \right)^{-1} \det(\h{M}) = \frac{1}{m!} \det(\h{M}).
\end{align}
$\h{M}$ is the difference of two simpler matrices
\begin{align}\notag
\h{M} =  \left( k^j \right)_{kj} - \left( (k-1)^j \right)_{kj}.
\end{align}
In particular, if we define 
\begin{align}\notag
V = \big( k^j \big)_{kj}, \quad T = \big(T_{kj} \big)_{kj},\quad \text{where}\  T_{kj} = \begin{cases}
1 & $ if $ j = k \\
-1 & $ if $ j = k-1 \\
0 & $ otherwise$.
\end{cases}
\end{align}
Then we have $\h{M} = VT$. Observe that $\det(T) = 1$. Also $V$ is a Vandermonde matrix whose determinant is thus given by
\begin{align}\notag
\det(V) = \prod_{1 \leq \ell < k \leq m} (k - \ell).
\end{align}
We therefore conclude that
\begin{align}\notag
\det(M) = \frac{1}{m!} \det(\h{M}) =\frac{1}{m!} \det(V) \det(T)
= \frac{1}{m!} \prod_{1 \leq \ell < k \leq m} (k - \ell).
\end{align}
Thus $\det(M) \neq 0$. Define a collection of polynomials $\Tht = \{\tht_k\}_{k=0}^{m-1} \subseteq \mathcal{P}_{m-1,1}$ with the coefficients of each given by the columns of the matrix $M^{-1}$. In particular, let the  coefficient of $x^j$ in $\tht_k$ be given by the $(j+1)$-th entry in the $(k+1)$-th column of $M^{-1}$
\begin{align}\notag
\tht_k(x) = \sum_{j=0}^{m-1} M^{-1}_{(j+1)(k+1)} x^j \quad \text{for}\ 0\leq k \leq m-1.
\end{align}
Since $M_{\ell j} = \s_{\ell-1}(x^{j-1})$ we have, for $0\leq \ell,k \leq m-1$,
\begin{align}\notag
\s_\ell(\tht_k) = \sum_{j=0}^{m-1}  \s_\ell(x^j) M^{-1}_{(j+1)(k+1)}= \sum_{j=0}^{m-1} M_{(\ell+1)(j+1)} M^{-1}_{(j+1)(k+1)}= \de^\ell_k.\notag
\end{align}
We conclude that there exists a basis $\Tht = \{\tht_\ell \}_{\ell =0}^{m-1}$ of $\mathcal{P}_{m-1,1}$ such that $\s_k(\tht_\ell) = \de_{k\ell}$ and therefore that the collection $\Sigma = \{\s_k \}_{k=0}^{m-1}$ forms a basis for $\mathcal{P}_{m-1,1}^*$.
\end{proof}
We proceed to the case of general dimension. 

\begin{Prop}  \label{lem:vol:uni}
Let $m,d \geq 1$. Then $\mathcal{P}_{m-1,d} = \left\{ \sum_{\al\in\mathcal{A}_m} p_\al x^\al \right\}$ is unisolvent with respect to the functionals $\Sigma = \{\s_\al \}_{\al\in\mathcal{A}_m}$ given by $\s_{\al}(\phi) = \int_{E_{\al}} \phi(x) dx$, and in particular there exists a unique set $\{\tht_{\be} \}_{\be\in\mathcal{A}_m} \subseteq \mathcal{P}_{m-1,n}$ such that $\s_{\al}(\tht_{\be}) = \de_{\al}^{\be}$.
\end{Prop}
\begin{proof}
Upon rescaling, we may apply \cref{lem:vol:uni1d}, to deduce that for $m\geq 0$ and $0 \leq k \leq m-1$, there exists a univariate polynomial $\tht_k \in \mathcal{P}_{m-1,1}$ such that
\begin{align}\notag
\int_{\frac{\ell}{m}}^{\frac{\ell+1}{m}} \tht_k(x) dx = \de_k^\ell.  
\end{align}
Let $n\geq 1$, and define for each $\al = (\al_1,..,\al_d) \in \mathcal{A}_m$ a polynomial $\tht_{\al} \in \mathcal{P}_{m-1,n}$
\begin{align}\notag
\tht_{\al}(x_1,..,x_d) = \prod_{i=1}^d \tht_{\al_i}(x_i).
\end{align}
We have
\begin{align}\notag
\s_{\be}(\tht_{\al}) = \int_{\frac{\be_1}{m}}^{\frac{\be_1+1}{m}} \cdots \int_{\frac{\be_d}{m}}^{\frac{\be_d+1}{m}} \prod_{i=1}^d \tht_{\al_i}(x_i) dx_n \cdots dx_1 = \prod_{i=1}^d \left( \int_{\frac{\be_i}{m}}^{\frac{\be_i+1}{m}} \tht_{\al_i}(x_i) dx_i \right)
= \prod_{i=1}^d \de_{\al_i}^{\be_i}.
\end{align}
Thus, given $m\in\N$, there exists a collection of polynomials $\Tht = \{ \tht_\al \}_{\al \in \mathcal{A}_m}$ such that
\begin{align}\notag
\s_{\be}(\tht_{\al}) = \de_{\al}^{\be}.
\end{align}
As $\text{dim}\, \mathcal{P}_{m-1,n} = \left| \Sigma \right|$, we conclude that $\Sigma$ forms a basis for $\mathcal{P}_{m-1,n}^*$ bi-orthogonal to $\Tht$ as constructed.
\end{proof}

We now define an operator $I_m : L^1_\text{loc} \to \mathcal{P}_{m-1,n}$ given by
\begin{align}\notag
I_m \phi(x) = \sum_{\al \in \mathcal{A}_m} \s_{\al}(\phi) \tht_{\al}(x).
\end{align}
By \cref{lem:vol:uni}
\begin{align}\notag
\s_{\be}(I_m \phi) = \int_{E_{\be}} I_m \phi(x) dx 
= \int_{E_{\be}} \sum_{\al\in\mathcal{A}_m} \s_{\al}(\phi) \tht_{\al}(x) dx
= \int_{E_{\be}} \s_{\be}(\phi) \tht_{\be}(x) dx 
= \s_{\be}(\phi).
\end{align}
Observe that $I_m$ is indeed a projection onto $\mathcal{P}_{m-1,n}$; for any polynomial $p \in \mathcal{P}_{m-1,n}$ and $\al \in \mathcal{A}_m$ we have $\s_{\al}(p) = \s_{\al}(I_m p)$, and therefore by \cref{lem:vol:uni} we have $p(x) = I_m p(x)$, as desired.

\newcommand{\etalchar}[1]{$^{#1}$}
\providecommand{\bysame}{\leavevmode\hbox to3em{\hrulefill}\thinspace}
\providecommand{\MR}{\relax\ifhmode\unskip\space\fi MR }
\providecommand{\MRhref}[2]{%
  \href{http://www.ams.org/mathscinet-getitem?mr=#1}{#2}
}
\providecommand{\href}[2]{#2}

\begin{multicols}{2}
\noindent Animikh Biswas\\ 
{\footnotesize
Department of Mathematics \& Statistics\\
University of Maryland--Baltimore County\\
Web: \url{https://userpages.umbc.edu/~abiswas/}\\
 Email: \url{abiswas@umbc.edu}} \\[.2cm]
\noindent

\noindent Vincent R. Martinez\\
{\footnotesize
Department of Mathematics \& Statistics\\
CUNY Hunter College\\
and\\
Department of Mathematics\\
CUNY Graduate Center\\
Web: \url{http://math.hunter.cuny.edu/vmartine/}\\
 Email: \url{vrmartinez@hunter.cuny.edu}
 }

\columnbreak 

\noindent Kenneth R. Brown\\ 
{\footnotesize
Department of Mathematics\\
University of California--Davis\\
 Email: \url{kbr@ucdavis.edu}}\\[.2cm]

\end{multicols}
\end{document}